\documentclass[review]{elsarticle}
\usepackage[usenames]{color}
\usepackage{extarrows}
\usepackage{hyperref}

\journal{Journal of \LaTeX\ Templates}

\usepackage{amsfonts}
\usepackage{amsmath, cases}
\usepackage{indentfirst}
\usepackage{graphicx}
\usepackage{latexsym,bm, amsthm} 

\usepackage{epsfig}
\usepackage{latexsym}
\usepackage{amssymb}
\usepackage{indentfirst}
\usepackage{amsmath}
\allowdisplaybreaks[2]%

\usepackage{xcolor}
\usepackage[normalem]{ulem} 
\newcommand\BJ{\bgroup\markoverwith
  {\textcolor{yellow}{\rule[-.5ex]{2pt}{3.5ex}}}\ULon}
\usepackage{geometry}  
\usepackage{tikz} 
\begin{document}
\begin{frontmatter}
\title{Dual minus partial order}

\author[1]{Ju Gao}
\ead{gaoju120085@163.com}

\author[1]{Hongxing Wang\corref{mycorrespondingauthor}}
\ead{winghongxing0902@163.com}

\author[2]{Xiaoji Liu}
\ead{xiaojiliu72@126.com}

\cortext[mycorrespondingauthor]{Corresponding author}
\address[1]{School of Mathematics and Physics, 
Guangxi Minzu University,
Nanning, 530006, China}
\address[2]{School of Education,
Guangxi Vocational Normal University,
Nanning, 530007,  China}

\begin{abstract}
 In this paper, we introduce the Dual-minus partial order, get some characterizations of the partial order, and  prove that both the dual star partial order and the dual sharp partial order are Dual-minus-type partial orders. Based on the Dual-minus partial order, we introduce  the Dual-minus sharp partial order and the Dual-minus star partial order, which are also Dual-minus-type partial orders. In addition, we discuss relationships among the Dual-minus sharp partial order, the D-sharp partial order and the G-sharp partial order(the Dual-minus star partial order, the D-star partial order and the P-star partial order).
\end{abstract}

\begin{keyword}
Dual-minus partial order;
Dual-minus star partial order;
Dual-minus sharp partial order;
dual star partial order;
dual sharp partial order
\MSC[2010]  15A09\sep  15A24 \sep 62G30
\end{keyword}
\end{frontmatter}


\section{Introduction}\label{Section-1-Introduction}
\newcommand{\rk}{{\rm rk}}
\newcommand{\Ind}{ {\rm Ind}}

\numberwithin{equation}{section}
\newtheorem{theorem}{T{\scriptsize HEOREM}}[section]
\newtheorem{lemma}[theorem]{L{\scriptsize  EMMA}}
\newtheorem{corollary}[theorem]{C{\scriptsize OROLLARY}}
\newtheorem{proposition}[theorem]{P{\scriptsize ROPOSITION}}
\newtheorem{remark}{R{\scriptsize  EMARK}}[section]
\newtheorem{definition}{D{\scriptsize  EFINITION}}[section]
\newtheorem{algorithm}{A{\scriptsize  LGORITHM}}[section]
\newtheorem{example}{E{\scriptsize  XAMPLE}}[section]
\newtheorem{problem}{P{\scriptsize  ROBLEM}}[section]
\newtheorem{assumption}{A{\scriptsize  SSUMPTION}}[section]

 In this paper,  $\mathbb{R}^{m\times n}$ is used to represent the set of $m\times n$ real matrices.
  Let $\rk(E)$ represent the rank of matrix $E$,
$E^{\rm {T}}$  represent
transpose of $E\in\mathbb{R}^{m\times n}$, ${\rm Ind}\left( E \right)$ represent the index of $E$ and $I_n$ be the identity matrix of $n$ order.
Furthermore, denote
$${\mathbb{R}}_n^{\tiny \mbox{\rm CM}}=\{E\mid E\in \mathbb{R}^{n\times n}, \  \rk\left(E^2 \right)=\rk \left(E \right)\}.$$
The symbol
$\varepsilon$ is used to represent the hypercomplex unit basis,
which satisfies $\varepsilon \neq 0$, $0 \varepsilon=\varepsilon 0 =0$,
$1 \varepsilon=\varepsilon 1 =\varepsilon$ and $\varepsilon^2=0$;
$\widehat E=E + \varepsilon E_0$ is used to represent an $m$-by-$n$ dual matrix,
where
$E$ and $E_0 \in\mathbb{R}^{m\times n}$ are the standard part and the dual part of $\widehat E$,  respectively.
The symbol  $\mathbb{D}^{m\times n}$  is used to represent the set of $m\times n$  dual matrices.
Recently, dual matrix theory and its applications attract much attention.
 Many researchers have yielded fruitful and interesting results
\cite{Stefanelli R E2007,Valentini P P 2018,Qi L Standard,Qi L vector,Qi L analysis,Udwadia F E2020,
wang2020mamt,gaoju,Wang2024pjo,Wang2023laaA,Wei2023arXiv,Wei2024coam}.

It is well known that  the Moore-Penrose inverse $E^\dagger$ of $E\in\mathbb{R}^{m \times n}$ is the unique matrix which  satisfies the Penrose
equations; see e.g. \cite{A2003-generalized inverse,2018-guangyini}:
  $EE^\dagger E = E$,   $E^\dagger EE^\dagger = E^\dagger$,   $(EE^\dagger)^{\rm T} = EE^\dagger$,   $(E^\dagger E)^{\rm T}= E^\dagger E$.
It is widely used in the study of singular linear systems.
In order to study singular dual linear systems, several dual generalized inverses are introduced.
This symbol $\widehat  E^{p}=E^{\dag}-\varepsilon E^{\dag}E_0E^{\dag}$ \cite{Stefanelli R E2007} is used to represent the Moore-Penrose dual generalized inverse (MPDGI for short)
of $\widehat  E=E+\varepsilon E_0\in\mathbb{D}^{m\times n}$.
 For any dual matrix, its MPDGI always exists.
Pennestr\`{\i} et al.\cite{Valentini P P 2018}
 give algorithms/formulas  about the
computation of the MPDGI.

If there exists a unique matrix $\widehat X\in\mathbb{D}^{n\times m}$ satisfying the Penrose equations:
\begin{align}
\label{DMPGI-Def}
 \widehat E\widehat X\widehat E=\widehat E, \
 \widehat X\widehat E\widehat X=\widehat X,\
 \widehat E\widehat X=\left(\widehat E\widehat X\right)^{  T},\
 \widehat X\widehat E=\left(\widehat X\widehat E\right)^{  T},
\end{align}
then $\widehat X$ is called the dual Moore-Penrose generalized inverse (DMPGI for short)
 of $\widehat E$ \cite{Udwadia F E2020},
and denoted  by $\widehat X=\widehat E^{\dag}$.
In \cite{Udwadia F E2020}, Udwadia et al.  show that not all dual matrices have DMPGIs, and get some interesting properties of the DMPGI.
In \cite{wang2020mamt}, Wang  gets that
the DMPGI of  $\widehat E$ exists if and only if $\left(I_m-EE^{\dag}\right)  E_0 (I_n-E^{\dag}E)=0$.

It is well known that  $E\in\mathbb{R}^{n\times n}$ is group invertible   if and only if the index of $E$ is 1.
The  group inverse $E^\#$ is  the unique matrix $X\in\mathbb{R}^{n\times n}$
satisfying
$ EXE=E$, $ XEX=X$ and  $EX=XE$ \cite{{A2003-generalized inverse}, {2018-guangyini}},
where  $E\in\mathbb{R}^{n\times n}$ and ${\rm Ind}(E)=1$.
In  \cite{zhong2022},
 Zhong and Zhang  introduce  the  dual group generalized inverse (DGGI for short).
Let $\widehat E \in\mathbb{D}^{n\times n}$,
 if a dual matrix $\widehat X \in\mathbb{D}^{n\times n}$ satisfies
\begin{align}
\nonumber
 \widehat E \widehat X \widehat E=\widehat E,\quad
 \widehat X \widehat E \widehat X=\widehat X,\quad
 \widehat E \widehat X=\widehat X \widehat E,
\end{align}
then $\widehat X$ is called DGGI of $\widehat E$,  and is denoted by $\widehat E^\#$.
Zhong and Zhang \cite{zhong2022}  also give
some equivalent conditions for the existence of   DGGI.
The DGGI of  $\widehat E$ exists if and only if
 ${\rm Ind}(E)=1$ and $\left(I_n-EE^{\#}\right)  E_0 (I_n-E^{\#}E)=0$, \cite{zhong2022}.
Due to the special structure of dual matrix,
\begin{align*}
 \widehat E^g = E^\#-\varepsilon E^\#E_0E^\#
  \end{align*}
 is called
the group  dual generalized inverse(GDGI for short)
of  $\widehat E$.
Obviously, $\widehat E^g$ always  exists  when the index of $E$ is one,
while
DGGI may not exist under the same restriction.
In \cite{gaoju}, Wang and Gao introduce the dual index one
and prove that the dual index ${\rm Ind}\left(\widehat{E}\right)$ of $\widehat{E}$ is one
if and only if
its DGGI exists.
  Denote
$$\mathbb{D}_n^{\tiny \mbox{\rm CM}}
=
\left\{\widehat{E} \left|
{\rm Ind}\left(\widehat{E}\right)=1, \ \widehat{E} \in\mathbb{D}^{n\times n}
\right.\right\}.$$
And it is easy to check that the DMPGI of any dual matrix in $\mathbb{D}_n^{\tiny \mbox{\rm CM}}$ exists.

As widely recognized, generalized inverse stands out as an essential tool for investigating matrix partial
orders. A matrix partial order, constituting a binary relation on a non-empty set, is characterized by its
reflexivity, transitivity, and antisymmetry properties.
 The matrix partial order theory is found widespread application diverse research domains,
 encompassing fields like data mining, engineering, genetics, physics, probability, and socioeconomic studies\cite{{2017-Fattore,2018Mostajeran,1992Peajcariaac,2008Simovici}}.
Recently,
Wang  et al.\cite{huangpei,tianhe} introduce several dual partial orders by using dual generalized inverses.

Let $\widehat E=E + \varepsilon E_0$,
$\widehat F=F + \varepsilon F_0\in\mathbb{D}^{m\times n}$,
and the DMPGIs of $\widehat  E$ and $\widehat  F$  exist. Then

(1) D-star partial order:
 $\widehat E\overset{\tiny\mbox{\rm  D\!-}\ast}\leq\widehat F$:
 $\widehat E^{\rm {T}}\widehat E=\widehat E^{\rm {T}}\widehat  F$,
 $\widehat E\widehat E^{\rm {T}}=\widehat F\widehat E^{\rm {T}}$;

(2) P-star partial order:
 $\widehat E\overset{\tiny\mbox{\rm  P\!-}{*}}\leq\widehat F$:
 $\widehat E^p\widehat E=\widehat E^p\widehat F$,
 $\widehat E\widehat E^p=\widehat F\widehat E^p$;

 Let $\widehat E=E + \varepsilon E_0,
 \widehat F=F + \varepsilon F_0\in\mathbb{D}_n^{\tiny \mbox{\rm CM}}$.
  Then

(3) D-sharp partial order:
 $\widehat E\overset{\tiny\mbox{\rm  D \!-}\#}\leq\widehat F$:
 $\widehat E^{\#}\widehat E=\widehat E^{\#}\widehat  F$,
 $\widehat E\widehat E^{\#}=\widehat F\widehat E^{\#}$;

 (4) G-sharp partial order:
 $\widehat E\overset{\tiny\mbox{\rm  G \!-}\#}\leq\widehat F$:
 $\widehat E^g\widehat E=\widehat E^g\widehat F$,
 $\widehat E\widehat E^g=\widehat F\widehat E^g$.

 As we know, in real(complex) field, minus partial order plays a crucial role,
and many famous matrix partial orders belong to minus-type partial orders,
such as the star partial order, the sharp partial order  and so on.
 In the paper, we will introduce a new binary  relation ---  the Dual-minus binary relation.
 Consider if it constitutes a matrix partial order.
If so,
 can the Dual-minus partial order include dual sharp partial orders and dual star partial orders?
 The rest of this paper is as follows:

In Section \ref{Section-2-Preliminaries},  we present some preliminary  knowledge.
In Section \ref{Section-3-Dual-minus-PO}, we introduce  the Dual-minus partial order, provide  its equivalent characterization,
  prove that the D-star, P-star, D-sharp, G-sharp partial orders are all Dual-minus-type partial orders,
and provide equivalent characterizations of these partial orders.
In Sections \ref{Section-4-Dual-minus-sharp-PO}(\ref{Section-5-Dual-minus-star-PO}), we introduce a new dual sharp(star) partial order ---  the Dual-minus sharp(star) partial order. While providing some equivalent characterizations of this partial order, we discuss its relationships with the D-sharp partial order and the G-sharp partial order(the D-star partial order and the P-star partial order).  The Dual-minus sharp(star) partial order belongs to the Dual-minus-type partial order and includes the D-sharp partial order and the G-sharp partial order(the D-star partial order and the P-star partial order).

\section{Preliminaries}\label{Section-2-Preliminaries}
In this section, we provide preliminary knowledge about dual generalized inverses and matrix partial orders
 including  the minus partial order, sharp partial order, star partial order,
 D-sharp partial order, G-sharp partial order, D-star partial order and P-star partial order.

\begin{lemma}[\cite{Mitra2010juzhenfenjie}]
\label{lemma-1}
Let $E, F \in \mathbb{R}^{n\times n}$,
 ${\rm \rk}\left( E \right)=r_e$ and ${\rm \rk}\left( F \right)=r_f$.
Then the following three conditions   are equivalent:
\begin{enumerate}
   \item[{\rm (1)}] $E \leq F$;
   \item[{\rm (2)}] $\rk\left(F-E\right)=\rk(F)-\rk(E)$;
   \item[{\rm (3)}] There exist orthogonal matrices $U$ and $V$ such that
   \begin{align}
   \label{minus}
   E=U\begin{pmatrix}
   D_1&O&O
   \\O&O&O
   \\O&O&O
   \end{pmatrix}V^{\rm{ T }}
   \ ,\
   F=U\begin{pmatrix}
   D_1+RD_2S& RD_2&O
   \\D_2S&D_2&O
   \\O&O&O
   \end{pmatrix}V^{\rm{ T }},
   \end{align}
   where $D_1\in \mathbb{R}^{r_e\times r_e}$ and $D_2\in \mathbb{R}^{(r_f-r_e)\times (r_f-r_e)}$ are invertible,
      $R\in \mathbb{R}^{r_e\times (r_f-r_e)}$ and $S\in \mathbb{R}^{(r_f-r_e)\times r_e}$ are some suitable matrices.
 \end{enumerate}
\end{lemma}

\begin{lemma}[\cite{Mitra2010juzhenfenjie}]
\label{xiaotuilun}
Let $E, F \in \mathbb{R}^{n\times n}$,
 ${\rm \rk}\left( E \right)=r_e$, ${\rm \rk}\left( F \right)=r_f$
  and $\Ind(E)=\Ind(F)=1$.
Then $E \leq F$ if and only if there is a nonsingular matrix $P$ such that
\begin{align}
   \label{TUIminus}
   E=P\begin{pmatrix}
   D_1&O&O
   \\O&O&O
   \\O&O&O
   \end{pmatrix}P^{-1}
   \ ,\
   F=P\begin{pmatrix}
   D_1+RD_2S& RD_2&O
   \\D_2S&D_2&O
   \\O&O&O
   \end{pmatrix}P^{-1},
   \end{align}
    where $D_1\in \mathbb{R}^{r_e\times r_e}$
   and $D_2\in \mathbb{R}^{(r_f-r_e)\times (r_f-r_e)}$ are invertible,
$R\in \mathbb{R}^{r_e\times (r_f-r_e)}$ and $S\in \mathbb{R}^{(r_f-r_e)\times r_e}$ are some suitable matrices.
\end{lemma}

\begin{lemma}[\cite{Baksalary J K2003,Mitra2010juzhenfenjie}]
\label{peiRStar-Def}
Let $E,F\in\mathbb{R}^{m\times n}$,
${\rk}\left( E \right)=r_e$ and
 ${\rk}\left( F \right)=r_f$.
Then the following five conditions are equivalent:
\begin{enumerate}
  \item[{\rm (1)}]
   $E\overset\ast\leqslant F$;
  \item[{\rm (2)}]
$E \leq F$,  $EF^{\rm {T}} =FE^{\rm{ T }}$ and $ F^{\rm{ T }}E=E^{\rm{ T }}F$;
  \item[{\rm (3)}]
$E^\dagger E=E^\dagger F$ and $ EE^\dagger=FE^\dagger $;
  \item[{\rm (4)}]
 $E^{\rm{ T }}E=E^{\rm{ T }}F$ and $EE^{\rm{ T }}=FE^{\rm{ T }}$;
  \item[{\rm (5)}]
There are orthogonal matrices $U$ and $V$ such that
\begin{align}
\label{2.1}
E=U\begin{pmatrix}D_1&O&O\\O&O&O\\O&O&O\end{pmatrix}V^{\rm{ T }}, \
F=U\begin{pmatrix}D_1&O&O\\O&D_2&O\\O&O&O\end{pmatrix}V^{\rm{ T }},
\end{align}
where
$D_1 \in \mathbb{R}^{r_e\times r_e}$
and
$D_2 \in \mathbb{R}^{(r_f-r_e)\times (r_f-r_e)}$
are invertible.
\end{enumerate}
\end{lemma}

\begin{lemma}
[\cite{Mitra1987-group-tichu,Mitra2010juzhenfenjie}]
\label{bbbbbb}
Let $E,F \in {\mathbb{R}}_n^{\tiny \mbox{\rm CM}}$,
 $ {\rk}\left( E \right)=r_e$ and ${\rk}\left( F \right)=r_f$.
 Then the following are equivalent:
 \begin{enumerate}
   \item[{\rm (1)}] $E\overset\#\leq F$;
   \item[{\rm (2)}] $E \leq F$ and  $EF=FE$;
   \item[{\rm (3)}] $E^\#E=E^\#F$ and $EE^\#=FE^\#$;
   \item[{\rm (4)}] $EF=E^2=FE$;
   \item[{\rm (5)}] There is a nonsingular matrix $P \in \mathbb{R}^{n\times n}$ such that
\begin{align}
\label{3.4}
E=P\begin{pmatrix}D_1&O&O\\O&O&O\\O&O&O\end{pmatrix}P^{-1}, \
F=P\begin{pmatrix}D_1&O&O\\O&D_2&O\\O&O&O\end{pmatrix}P^{-1},
\end{align}
 where $D_1\in \mathbb{R}^{r_e\times r_e}$ and $D_2\in \mathbb{R}^{(r_f-r_e)\times (r_f-r_e)}$ are invertible.
 \end{enumerate}
\end{lemma}

\begin{lemma}[\cite{huangpei}]
\label{PEIDSPO-Char-2-Th}
Let $\widehat  E=E+\varepsilon E_0\in\mathbb{D}^{m\times n}$,
$\widehat  F=F+\varepsilon F_0\in\mathbb{D}^{m\times n}$,
${\rm \rk}\left( E \right)=r_e$, ${\rm \rk}\left( F \right)=r_f$
and the DMPGIs of $\widehat  E$ and $\widehat F$ exist.
Then the following are equivalent:
 \begin{enumerate}
   \item[{\rm (1)}]  $\widehat E\overset{\tiny\mbox{\rm  D\!-}\ast}\leq\widehat F$;
   \item[{\rm (2)}] $\widehat E^{\dag}\widehat E=\widehat E^{\dag}\widehat  F$, \
                    $\widehat E\widehat E^{\dag}=\widehat F\widehat E^{\dag}$;
   \item[{\rm (3)}] $\widehat E^{\rm{ T }}\widehat E=\widehat E^{\rm{ T }}\widehat  F , \
  \widehat E\widehat E^{\rm{ T }}=\widehat F\widehat E^{\rm{ T }}$;
   \item[{\rm (4)}] There are orthogonal matrices $U$ and $V$
such that
{\small
\begin{align}
\label{PEIDSPO-Char-2}
\left\{\begin{array}{l}\widehat E
=U\begin{pmatrix}
D_1&O&O\\
O&O&O\\
O&O&O\end{pmatrix}V^{\rm{ T }}
+\varepsilon U
\begin{pmatrix}
E_1&E_2&E_3\\
E_4&O&O\\
E_7&O&O
\end{pmatrix}V^{\rm{ T }},
\\
\widehat F=
U\begin{pmatrix}
 D_1&O&O\\
O& D_2&O\\
O&O&O
\end{pmatrix}
V^{\rm{ T }}
+\varepsilon U\begin{pmatrix}
E_1&E_2-{ D_1^{-1}}{  E_4^{\rm {T}}D_2}&E_3\\
E_4-D_2{E_2^{\rm {T}}}{ D_1^{-1}}&F_5&F_6\\
E_7&F_8&O
\end{pmatrix}V^{\rm{ T }},
\end{array}\right.
\end{align}}
   where $D_1\in \mathbb{R}^{r_e\times r_e}$
   and
   $D_2\in \mathbb{R}^{(r_f-r_e)\times (r_f-r_e)}$ are  invertible.
 \end{enumerate}
\end{lemma}

  \begin{lemma}[\cite{huangpei}]
 \label{peiDPSO-Char-1-Th}
 Let $\widehat  E=E+\varepsilon E_0\in\mathbb{D}^{m\times n}$,
$\widehat  F=F+\varepsilon F_0\in\mathbb{D}^{m\times n}$,
${\rm \rk}\left( E \right)=r_e$, ${\rm \rk}\left( F \right)=r_f$
and the DMPGIs of $\widehat  E$ and $\widehat  F$ exist.
Then the following are equivalent:
 \begin{enumerate}
   \item[{\rm (1)}]  $\widehat E\overset{\tiny\mbox{\rm  P\!-}{*}}\leq\widehat F$;
   \item[{\rm (2)}] $\widehat E^p\widehat E=\widehat E^p\widehat F$, \
                     $\widehat E\widehat E^p=\widehat F\widehat E^p$;
   \item[{\rm (3)}] $E^{\rm T} E =  E^{\rm{ T }}F$, \
                    $E E^{\rm{ T }} =F E^{\rm{ T }}$,
                  \\
                   $E^{\rm{ T }}E_0= E^{\rm{ T }}F_0$, \
                   $E_0E^{\rm{ T }}=F_0E^{\rm{ T }}$;
   \item[{\rm (4)}] There are orthogonal matrices $U$ and $V$   such that
\begin{align}
 \label{peiDPSO-Char-1}
\left\{\begin{array}{l}\widehat E
=U\begin{pmatrix}    D_1&O&O\\O&O&O\\O&O&O\end{pmatrix}V^{\rm{ T }}
+\varepsilon U\begin{pmatrix}
E_1&E_2&E_3\\E_4&O&O\\E_7&O&O\end{pmatrix}V^{\rm{ T }},\\
\widehat F
=U\begin{pmatrix}    D_1&O&O\\O&    D_2&O\\O&O&O
\end{pmatrix}V^{\rm{ T }}
+
\varepsilon U\begin{pmatrix}
E_1&E_2&E_3\\E_4&F_5&F_6\\E_7&F_8&O
\end{pmatrix}V^{\rm{ T }},
\end{array}\right.
 \end{align}
 where $D_1\in \mathbb{R}^{r_e\times r_e}$ and $D_2\in \mathbb{R}^{(r_f-r_e)\times (r_f-r_e)}$ are invertible.
    \end{enumerate}
\end{lemma}

From Lemma \ref{PEIDSPO-Char-2-Th} and Lemma \ref{peiDPSO-Char-1-Th},
it is easy to see that
the D-star partial order and the P-star partial order are different.
Interestingly, if $\widehat E^{\dag}=\widehat E^p$,
then the D-star partial order coincides with the P-star partial order.

\begin{lemma}[\cite{tianhe}]
\label{TIANHEDSPO-Char-2-Th}
 Let $\widehat  E=E+\varepsilon E_0$,
$\widehat  F=F+\varepsilon F_0\in\mathbb{D}_n^{\tiny \mbox{\rm CM}}$,
${\rm \rk}\left( E \right)=r_e$ and ${\rm \rk}\left( F \right)=r_f$.
 Then the following are equivalent:
 \begin{enumerate}
   \item[{\rm (1)}]  $\widehat E\overset{\tiny\mbox{\rm  D \!-}\#}\leq\widehat F$;
   \item[{\rm (2)}] $\widehat E^{\#}\widehat E=\widehat E^{\#}\widehat F$,
                  \ $\widehat E\widehat E^{\#}=\widehat F\widehat E^{\#}$;
   \item[{\rm (3)}] $EF=E^2=FE$,\ $EF_0+E_0F=E_0E+EE_0=FE_0+F_0E$;
   \item[{\rm (4)}]There is
 a nonsingular matrix $P$ such that
{\small
\begin{align}
\label{pei-DSPO-Char-2}
\left\{\begin{array}{l}\widehat E
=P\begin{pmatrix}
D_1&O&O\\
O&O&O\\
O&O&O\end{pmatrix}P^{-1}
+\varepsilon P
\begin{pmatrix}
E_1&E_2&E_3\\
E_4&O&O\\
E_7&O&O
\end{pmatrix}P^{-1},\\
\widehat F=
P\begin{pmatrix}D_1&O&O\\
O&D_2&O\\
O&O&O
\end{pmatrix}
P^{-1}
+\varepsilon P\begin{pmatrix}E_1
&E_2-{D_1^{-1}}E_2D_2&E_3\\
E_4-D_2E_4D_1^{-1}&
F_5&F_6\\E_7&F_8&O\end{pmatrix}P^{-1},
\end{array}\right.
\end{align}}
where
$D_1\in \mathbb{R}^{r_e\times r_e}$
and
$D_2\in \mathbb{R}^{(r_f-r_e)\times (r_f-r_e)}$ are invertible.
\end{enumerate}
\end{lemma}

\begin{lemma}[\cite{tianhe}]
\label{tianheDPSPO-Char-1-Th}
  Let $\widehat  E=E+\varepsilon E_0$,
$\widehat  F=F+\varepsilon F_0\in\mathbb{D}_n^{\tiny \mbox{\rm CM}}$,
  ${\rm \rk}\left( E \right)=r_e$ and ${\rm \rk}\left( F \right)=r_f.$
Then the following are equivalent:
 \begin{enumerate}
   \item[{\rm (1)}] $\widehat E\overset{\tiny\mbox{\rm  G \!-}\#}\leq\widehat F$;
   \item[{\rm (2)}] $\widehat E^g\widehat E=\widehat E^g\widehat F$, \
                    $\widehat E\widehat E^g=\widehat F\widehat E^g$;
   \item[{\rm (3)}] $EF=E^2=FE$,\ $EE_0=EF_0$,\ $E_0E=F_0E$;
   \item[{\rm (4)}] There is a nonsingular matrix $P$ such that
\begin{align}
\label{DPSPO-Char-1-1}
\left\{\begin{array}{l}
\widehat E
=
P\begin{pmatrix}
D_1&O&O\\
O&O&O\\
O&O&O\end{pmatrix}
P^{-1}
+\varepsilon P
\begin{pmatrix}
E_1&E_2&E_3\\E_4&O&O\\E_7&O&O
\end{pmatrix}P^{-1},\\
\widehat F
=
P\begin{pmatrix}
  D_1&O&O\\
  O& D_2&O\\
  O&O&O\end{pmatrix}P^{-1}
+\varepsilon P\begin{pmatrix}
E_1&E_2&E_3\\
E_4& F_5&F_6\\
E_7&F_8&O
\end{pmatrix}P^{-1},
\end{array}\right.
\end{align}
where
$D_1\in \mathbb{R}^{r_e\times r_e}$
and
 $D_2\in \mathbb{R}^{(r_f-r_e)\times (r_f-r_e)}$ are invertible.
\end{enumerate}
\end{lemma}

These two   types of dual sharp partial  orders are not equivalent and their differences mainly lie in the dual part.
When $ EF_0=EE_0=FE_0$
  and
  $ E_0F=E_0E=F_0E$, the D-sharp partial order is equivalent to the G-sharp partial order.

\begin{lemma}
[{\cite{wang2020mamt}}]
\label{Wang2021MAMT-Th2.1}
Let $\widehat E=E+\varepsilon E_0 \in \mathbb{D}^{m\times n}$.
Then the DMPGI $\widehat E^{\dag}$ of $\widehat E$ exists
if and only if
 \begin{align}
 \label{good-one}
 {\rk}\left(\begin{matrix}
  E_0  &E\\ E  &O
  \end{matrix}\right)
  =2{\rk}\left( E \right).
  \end{align}
\end{lemma}

Let $ \widehat E=E+\varepsilon E_0$ and
denote
\begin{align}
\label{D-Rank}
\rk\left( \widehat E \right)
=
\rk\left(\begin{matrix}
 E_0  &     E           \\
 E      &   O
\end{matrix}\right)-\rk (E).
\end{align}

By applying Lemma \ref{Wang2021MAMT-Th2.1},
it is easy to get the following Lemma \ref{Th-Rank=DMPGI}.
\begin{lemma}
\label{Th-Rank=DMPGI}
Let
$\widehat E=E+\varepsilon E_0$.
Then the DMPGI of $\widehat E$   exists
if and only if
\begin{align}
\label{Th-Rank=DMPGI-Eq}
\rk\left( \widehat E \right)=\rk(E).
\end{align}
\end{lemma}

\section{Dual-minus Partial Order}
\label{Section-3-Dual-minus-PO}
In this section,
 we introduce the dual-minus partial order by applying  the minus partial order and the DMPGI,
and discuss relationships between the
 Dual-minus partial order and the  D-sharp partial order,
 D-star partial order,
 G-sharp partial order,  P-star partial order respectively.

\begin{definition}
\label{Dual-minusPartialOrder-Def}
Let
$\widehat E=E+\varepsilon E_0$,
$\widehat F=F+\varepsilon F_0\in\mathbb{D}^{m\times n}$,
and $\widehat E$ and $\widehat F$ have the DMPGIs.
Write $\widehat E\overset{\tiny\mbox{\rm D\!}}\leq\widehat F$:
\begin{align}
\label{nice-def}
E \leq F \
\mbox{and the DMPGI of} \ \widehat F-\widehat E \ \mbox{exists}.
\end{align}
We call $\widehat E$ is below $\widehat F$  under the Dual-minus order.
 \end{definition}

\begin{theorem}
\label{nice-1}
 Let $\widehat E=E+\varepsilon E_0$,
$\widehat F=F+\varepsilon F_0\in\mathbb{D}^{m\times n}$,
and $\widehat E$ and $\widehat F$ have the DMPGIs.
Then $\widehat E\overset{\tiny\mbox{\rm D\!}}\leq\widehat F$  if and only if
\begin{align}
\label{ok-1}
\rk\left( \widehat F \right)-\rk\left( \widehat E \right)
=\rk\left( \widehat F-\widehat E \right).
\end{align}
\end{theorem}

\begin{proof}
Since $\widehat E$ and $\widehat F$ have the DMPGIs,
 we have
$\rk\left( \widehat E \right)=\rk\left(  E \right)$
and
$\rk\left( \widehat F \right)=\rk\left(  F \right)$ according to Lemma \ref{Th-Rank=DMPGI}.

Let $\widehat E\overset{\tiny\mbox{\rm D\!}}\leq\widehat F$. 
Then
 $\widehat F-\widehat E$ has the DMPGI
 and $E\leq F$ from Definition \ref{Dual-minusPartialOrder-Def}.
Since the DMPGI of $\widehat F-\widehat E$  exists,
we have
$\rk\left( \widehat F-\widehat E \right)=
\rk\left(  F - E \right)$ by applying Lemma \ref{Th-Rank=DMPGI}.
 And since
$E\leq F$,
 we have $\rk\left(  F \right)-\rk\left(  E \right)=\rk\left(  F-E \right)$ by applying Lemma \ref{lemma-1}.
Thus,
$\rk\left( \widehat F \right)-\rk\left( \widehat E \right)=\rk\left( \widehat F-\widehat E \right)$.

Conversely,
let $\rk\left( \widehat F \right)-\rk\left( \widehat E \right)=\rk\left( \widehat F-\widehat E \right)$.
Since
\begin{align}
&\rk\left( \widehat F \right)-\rk\left( \widehat E \right)=\rk\left( F\right)-\rk\left( E \right) \leq \rk\left( F-E \right),
\\
&\rk \left(  \widehat F-\widehat E  \right) \geq \rk \left( F-E \right),
\end{align}
we have $\rk \left(  \widehat F-\widehat E  \right) = \rk \left( F-E \right)$ and
 $\rk\left(  F \right)-\rk\left(  E \right)=\rk\left(  F- E \right)$.
 Therefore, we get that the  DMPGI of $\widehat F-\widehat E$ exists and
 $E\leq F$,
 that is, $\widehat E\overset{\tiny\mbox{\rm D\!}}\leq\widehat F$.
\end{proof}

 \begin{theorem}
 \label{nice-2}
 Let $\widehat E=E+\varepsilon E_0$,
$\widehat F=F+\varepsilon F_0\in\mathbb{D}^{m\times n}$,
and $\widehat E$ and $\widehat F$ have the DMPGIs.
Then $\widehat E\overset{\tiny\mbox{\rm D\!}}\leq\widehat F$ if and only if
\begin{align}
\label{ok-2}
\begin{pmatrix}
E_0& E
\\
E&  O
\end{pmatrix}
\leq
\begin{pmatrix}
F_0& F
\\
F&   O
\end{pmatrix}.
\end{align}
\end{theorem}

\begin{proof}
Let  the DMPGIs  of
$\widehat E$ and $\widehat F$ exist
  and  $\widehat E\overset{\tiny\mbox{\rm D\!}}\leq\widehat F$.
According to Lemma \ref{lemma-1},
Lemma \ref{Wang2021MAMT-Th2.1}
 and Definition \ref{Dual-minusPartialOrder-Def},
we have
 $\rk\left(F-E\right)=\rk(F)-\rk(E)$,
 $\rk\left(
\begin{smallmatrix}
E_0&E\\
E&O
\end{smallmatrix}\right)=2\rk(E)$,
 $\rk\left(
\begin{smallmatrix}
F_0&F\\
F&O
\end{smallmatrix}\right)=2\rk(F)$ and
$\rk\left(
\begin{smallmatrix}
F_0-E_0&F-E\\
F-E&O
\end{smallmatrix}\right)=2\rk(F-E)$.
Therefore,
$
\rk
\left(\begin{smallmatrix}
F_0&F\\
F&O
\end{smallmatrix}\right)-
\rk
\left(\begin{smallmatrix}
E_0&E\\
E&O
\end{smallmatrix}\right)=
2\rk(F)-2\rk(E)=
2\rk(F-E)=
\rk
\left(\begin{smallmatrix}
F_0-E_0&F-E\\
F-E&O
\end{smallmatrix}\right).$
It follows from Lemma \ref{lemma-1} that we get (\ref{ok-2}).

Conversely, by applying (\ref{ok-2}), we have
\begin{align}
\label{proof-nice2}
\rk
\begin{pmatrix}
F_0&F\\
F&O
\end{pmatrix}-
\rk
\begin{pmatrix}
E_0&E\\
E&O
\end{pmatrix}=
\rk
\begin{pmatrix}
F_0-E_0&F-E\\
F-E&O
\end{pmatrix}.
\end{align}
Since $\widehat E$ and $\widehat F$ have the DMPGIs, 
we have
$\rk
\left(\begin{smallmatrix}
E_0&E\\
E&O
\end{smallmatrix}\right)
 =2\rk(E)$ and
$\rk
\left(\begin{smallmatrix}
F_0&F\\
F&O
\end{smallmatrix}\right)=2\rk(F).$
From (\ref{proof-nice2}), it follows that
\begin{align}
\label{proof-nice22}
\rk
\begin{pmatrix}
F_0-E_0&F-E\\
F-E&O
\end{pmatrix}=2\left(\rk(F)-\rk(E)\right).
\end{align}
Furthermore, applying
\begin{align*}
2\left(\rk(F)-\rk(E)\right)\leq2\rk \left(F-E\right)\leq
\rk
\begin{pmatrix}
F_0-E_0&F-E\\
F-E&O
\end{pmatrix}=2\left(\rk(F)-\rk(E)\right),
\end{align*}
 we get
\begin{align*}
\rk(F)-\rk(E)=\rk(F-E)~
\mbox{and}~
\rk
\begin{pmatrix}
F_0-E_0&F-E\\
F-E&O
\end{pmatrix}=
2\rk\left(F-E\right).
\end{align*}
Therefore, $E\leq F$ and the DMPGI of $\widehat F-\widehat E$ exists,
 that is, $\widehat E \overset{\tiny\mbox{\rm D \!}}\leq \widehat F$.
\end{proof}

Applying Theorem \ref{nice-2}, we obviously obtain the following Theorem \ref{nice-3}.

\begin{theorem}
\label{nice-3}
The Dual-minus order is a  partial order.
\end{theorem}

\begin{remark}
Let $\widehat  E=E+\varepsilon E_0\in\mathbb{D}^{m\times n}$,
$\widehat  F=F+\varepsilon F_0\in\mathbb{D}^{m\times n}$,
and $\widehat  E$ and $\widehat F$ have the DMPGIs.
If $\widehat E\overset{\tiny\mbox{\rm  D\!-}\ast}\leq\widehat F$,
according to Lemma \ref{PEIDSPO-Char-2-Th},
 then we have the forms of $\widehat E$ and $\widehat F$ as in (\ref{PEIDSPO-Char-2}).
 The standard part $E$ of $\widehat E$
and
the standard part $F$ of $\widehat F$ in equation (\ref{PEIDSPO-Char-2}) clearly satisfy $E\leq F$.
Since \begin{align}
\label{see-one}
\widehat F-\widehat E
 =\left(F-E\right)+\varepsilon \left(F_0-E_0\right)=
U\begin{pmatrix}
O&O&O\\
O&D_2&O\\
O&O&O
\end{pmatrix}
V^{\rm{ T }}
+\varepsilon U\begin{pmatrix}
O&-D_1^{-1}{E_4^{\rm {T}}}D_2&O\\
-D_2{E_2^{\rm {T}}}{ D_1^{-1}}&F_5&F_6\\
O&F_8&O\end{pmatrix}V^{\rm{ T }},
\end{align}
then
 \begin{align*}
 {\rk}\left(\begin{matrix}
  F_0-E_0  &F-E\\ F-E  &O
  \end{matrix}\right)
  =2{\rk}\left( F-E \right).
  \end{align*}
  Thus, the DMPGI of $\widehat F-\widehat E$ exists from
 Theorem \ref{Wang2021MAMT-Th2.1}.
  Therefore, according to Definition
\ref{Dual-minusPartialOrder-Def},  the D-star partial order is a Dual-minus-type partial order.

Similarly,
by applying
Lemma \ref{peiDPSO-Char-1-Th},
 Lemma \ref{TIANHEDSPO-Char-2-Th},
 Lemma \ref{tianheDPSPO-Char-1-Th},
 Lemma \ref{Th-Rank=DMPGI}, and
 Definition \ref{Dual-minusPartialOrder-Def},
 we get that the
  P-star partial order, D-sharp partial order,
and G-sharp partial order are all Dual-minus-type partial orders.
\end{remark}

 Next, we provide a characterization of the Dual-minus partial order.

\begin{theorem}
\label{fenjie-the}
Let
$\widehat E=E+\varepsilon E_0 \in\mathbb{D}^{m\times n}$
and  $\widehat F=F+\varepsilon F_0 \in\mathbb{D}^{m\times n}$,
and $\widehat E$ and $\widehat F$ have the DMPGIs.
 Then $\widehat E\overset{\tiny\mbox{\rm  D \!}}\leq\widehat F$
  if and only if there are orthogonal matrices $U$ and $V$ such that
{\small\begin{align}
\label{fenjieshiA}
&\widehat E=
U\begin{pmatrix}
D_1&O&O\\
O&O&O\\
O&O&O\end{pmatrix}V^{\rm{ T }}
+\varepsilon U
\begin{pmatrix}
E_1&E_2&E_3\\
E_4&O&O\\
E_7&O&O
\end{pmatrix}V^{\rm{ T }},
\\
\label{fenjieshi-B}
&\widehat F=
U\begin{pmatrix}
D_1+RD_2S&RD_2&O\\
D_2S&D_2&O\\
O&O&O
\end{pmatrix}
V^{\rm{ T }}
+\varepsilon U\begin{pmatrix}
E_1+RM+NS{ -RF_5S}&{ E_2+N}&E_3+RF_6\\
{ E_4+M}&F_5&F_6\\
E_7+F_8S&F_8&O\end{pmatrix}V^{\rm{ T }},
\end{align}}
 where  $D_1\in \mathbb{R}^{r_e\times r_e}$, $D_2\in \mathbb{R}^{(r_f-r_e)\times (r_f-r_e)}$ are invertible, 
 and  $M,  S\in\mathbb{R}^{(r_f-r_e)\times r_e}$, $N, R\in\mathbb{R}^{r_e\times (r_f-r_e)}$  are some suitable matrices.
\end{theorem}

\begin{proof}
For $\widehat E\overset{\tiny\mbox{\rm  D \!}}\leq\widehat F$, we have $E\leq F$.
It follows from Lemma \ref{lemma-1} that $E$ and $F$  are as in (\ref{minus}).
Let $\widehat E$ and $\widehat F$ have the DMPGIs, and  denote
\begin{align}
\nonumber
&\widehat E=
U\begin{pmatrix}
D_1&O&O\\
O&O&O\\
O&O&O\end{pmatrix}V^{\rm{ T }}
+\varepsilon U
\begin{pmatrix}
E_1&E_2&E_3\\
E_4&O&O\\
E_7&O&O
\end{pmatrix}V^{\rm{ T }},
\\
\nonumber
&\widehat F=
U\begin{pmatrix}
D_1+RD_2S&RD_2&O\\
D_2S&D_2&O\\
O&O&O
\end{pmatrix}
V^{\rm{ T }}
+\varepsilon U\begin{pmatrix}
F_1&F_2&F_3\\
F_4&F_5&F_6\\
F_7&F_8&O\end{pmatrix}V^{\rm{ T }},
\end{align}
   where $E_1, F_1\in \mathbb{R}^{r_e\times r_e}$ and $F_5\in \mathbb{R}^{(r_f-r_e)\times (r_f-r_e)}$.
Then
\begin{align}
\label{okk}
&\widehat F-\widehat E=\left(F-E\right)+\varepsilon \left(F_0-E_0\right)
\\
\nonumber
&=
U\begin{pmatrix}
RD_2S&RD_2&O\\
D_2S&D_2&O\\
O&O&O
\end{pmatrix}
V^{\rm{ T }}
+\varepsilon U\begin{pmatrix}
F_1-E_1&F_2-E_2&F_3-E_3\\
F_4-E_4&F_5&F_6\\
F_7-E_7&F_8&O\end{pmatrix}V^{\rm{ T }}.
\end{align}
For
$\widehat E\overset{\tiny\mbox{\rm  D \!}}\leq\widehat F$,
 it follows from Definition \ref{Dual-minusPartialOrder-Def}
that  the DMPGI of $\widehat F-\widehat E$ exists.
By  applying Lemma \ref{Wang2021MAMT-Th2.1},
we have
\begin{align}
\label{okokok}
\rk
\begin{pmatrix}
F_0-E_0&F-E\\
F-E&O
\end{pmatrix}=2\rk\left(F-E\right).
\end{align}

Write  $M =F_4-E_4$ and
$ N =F_2-E_2$.
 Substituting (\ref{okk}) into (\ref{okokok}), we get
{\small
\begin{align*}
&\rk
\begin{pmatrix}
F_1-E_1&F_2-E_2&F_3-E_3&RD_2S&RD_2&O\\
F_4-E_4&F_5&F_6&D_2S&D_2&O\\
F_7-E_7&F_8&O&O&O&O\\
RD_2S&RD_2&O&O&O&O\\
D_2S&D_2&O&O&O&O\\
O&O&O&O&O&O
\end{pmatrix}
\\
\nonumber
&=
\rk
\begin{pmatrix}
F_1-E_1-RM-NS +RF_5S&O&F_3-E_3-RF_6&O&O&O\\
O&O&O&O&D_2&O\\
F_7-E_7-F_8S&O&O&O&O&O\\
O&O&O&O&O&O\\
O&D_2&O&O&O&O\\
O&O&O&O&O&O
\end{pmatrix}
=
2\rk\left(D_2\right).
\end{align*}}
Then applying (\ref{okokok}) gives
\begin{align*}
& F_1=E_1+RM+NS-RF_5S,\
 F_3=E_3+RF_6,\
 F_7=E_7+F_8S.
\end{align*}
Therefore, we get the form  (\ref{fenjieshi-B}) of $\widehat F$.

On the contrary,
we have $E \leq F$ and (\ref{okokok})  by  (\ref{fenjieshiA}) and  (\ref{fenjieshi-B}).
 According to  Lemma \ref{Wang2021MAMT-Th2.1},
 equation (\ref{okokok}) indicates the existence of DMPGI of $\widehat F-\widehat E$.
  Therefore,
$\widehat E\overset{\tiny\mbox{\rm D\!}}\leq\widehat F$
by Definition \ref{Dual-minusPartialOrder-Def}.
\end{proof}

Next, we  study relationships between  the Dual-minus partial order and the  D-star(P-star) partial order.

\begin{corollary}
\label{tui1}
Comparing Theorem  \ref{fenjie-the} with
Lemma \ref{PEIDSPO-Char-2-Th}
and Lemma \ref{peiDPSO-Char-1-Th} respectively,
we get

 (1) In Theorem \ref{fenjie-the}, giving
  $R=O$, $S=O$,  $M=-D_2E_2^{\rm{ T }}D_1^{-1}$ and $N=-D_1^{-1}E_4^{\rm{ T }}D_2$,
it is obvious that  the form of $\widehat F$ in (\ref{fenjieshi-B}) is the same as that in (\ref{PEIDSPO-Char-2}),
Therefore, $\widehat E\overset{\tiny\mbox{\rm  D\!-}\ast}\leq\widehat F$.

(2) In Theorem \ref{fenjie-the}, giving
$R=O$, $ S=O$, $M=O$ and $N=O$,
it is obvious that  the form of $\widehat F$ in (\ref{fenjieshi-B}) is the same as that in (\ref{peiDPSO-Char-1}),
Therefore, $\widehat E\overset{\tiny\mbox{\rm  P\!-}{*}}\leq\widehat F$.
\end{corollary}

In the following Theorem \ref{The-nice1} and Theorem \ref{The-nice2},
we apply the Dual-minus partial order to provide characterizations of  the D-star and P-star partial orders.

\begin{theorem}
\label{The-nice1}
Let $\widehat E=E+\varepsilon E_0$,
$\widehat F=F+\varepsilon F_0\in\mathbb{D}^{m\times n}$,
and $\widehat E$ and $\widehat F$ have the DMPGIs.
 Then
$\widehat E\overset{\tiny\mbox{\rm  D\!-}\ast}\leq\widehat F$
  if and only if
 $\widehat E\overset{\tiny\mbox{\rm D\!}}\leq\widehat F$
 and
 \begin{align}
 \label{go}
 \begin{aligned}
EF^{\rm T} &=\left(EF^{\rm T} \right)^{\rm T}, \
 F^{\rm T}E= \left(F^{\rm T}E \right)^{\rm T},
 \\
EE^\dagger\left(F_0-E_0\right)&=\left(E^{\rm{ T }}\right)^\dagger E_0^{\rm{ T }}\left(E-F\right),\
\left(F_0-E_0\right)E^\dagger E=\left(E-F\right)E_0^{\rm{ T }}\left(E^{\rm{ T }}\right)^\dagger.
\end{aligned}
\end{align}
\end{theorem}

\begin{proof}
``$\Leftarrow$'' \
Let $\widehat E=E+\varepsilon E_0$,
$\widehat F=F+\varepsilon F_0\in\mathbb{D}^{m\times n}$,
 $\widehat E\overset{\tiny\mbox{\rm D\!}}\leq\widehat F$
 and the forms of  $\widehat E$ and $\widehat F$ be as given in Theorem \ref{fenjie-the}, 
 where ${\rm \rk}\left( E \right)=r_e$ and ${\rm \rk}\left( F \right)=r_f$.

Since $\widehat E\overset{\tiny\mbox{\rm D\!}}\leq\widehat F$,  
we have $E \leq F$.
 Applying $EF^{\rm T}=\left(EF^{\rm T} \right)^{\rm T}$ and
$F^{\rm T}E=\left(F^{\rm T}E \right)^{\rm T}$,
it follows from Lemma \ref{peiRStar-Def}
that  $E\overset\ast\leq F$.
Furthermore, applying  (\ref{2.1}) gives   $S=O$ and $R=O$.
Then
\begin{align}
\label{fenjieshiAA}
&\widehat E=
U\begin{pmatrix}
D_1&O&O\\
O&O&O\\
O&O&O\end{pmatrix}V^{\rm{ T }}
+\varepsilon U
\begin{pmatrix}
E_1&E_2&E_3\\
E_4&O&O\\
E_7&O&O
\end{pmatrix}V^{\rm{ T }},
\\
\label{fenjieshiAB}
&\widehat F=
U\begin{pmatrix}
D_1&O&O\\
O&D_2&O\\
O&O&O
\end{pmatrix}
V^{\rm{ T }}
+\varepsilon U\begin{pmatrix}
E_1&N+E_2&E_3\\
M+E_4&F_5&F_6\\
E_7&F_8&O\end{pmatrix}V^{\rm{ T }},
\end{align}
where $M$, $N$, $F_5$, $F_6$ and $F_8$ are matrices of appropriate orders.

Substituting (\ref{fenjieshiAA}) and (\ref{fenjieshiAB})
 into
 $EE^\dagger\left(F_0-E_0\right)=\left(E^{\rm {\rm T}}\right)^\dagger E_0^{\rm T}\left(E-F\right)$
 and
$\left(F_0-E_0\right)E^\dagger E=\left(E-F\right)E_0^{{\rm T}}\left(E^{{\rm T}}\right)^\dagger$
gives
\begin{align*}
U\begin{pmatrix}
 I_{{\rk}\left( E \right)}&O&O\\
O&O&O\\
O&O&O
\end{pmatrix}
\begin{pmatrix}
O&N&O\\
M&F_5&F_6\\
O&F_8&O\end{pmatrix}V^{\rm T}
&=
U\begin{pmatrix}
D_1^{-1}&O&O\\
O&O&O\\
O&O&O
\end{pmatrix}
\begin{pmatrix}
E_1^{{\rm T}}&E_4^{{\rm T}}&E_7^{{\rm T}}\\
E_2^{{\rm T}}&O&O\\
E_3^{{\rm T}}&O&O
\end{pmatrix}
\begin{pmatrix}
O&O&O\\
O&{ -D_2}&O\\
O&O&O
\end{pmatrix}V^{\rm T},
\\
U\begin{pmatrix}
O&N&O\\
M&F_5&F_6\\
O&F_8&O\end{pmatrix}
\begin{pmatrix}
 I_{{\rk}\left( E \right)}&O&O\\
O&O&O\\
O&O&O
\end{pmatrix} V^{\rm T}
&
=
U\begin{pmatrix}
O&O&O\\
O&-D_2&O\\
O&O&O
\end{pmatrix}
\begin{pmatrix}
E_1^{{\rm T}}&E_4^{{\rm T}}&E_7^{{\rm T}}\\
E_2^{{\rm T}}&O&O\\
E_3^{{\rm T}}&O&O
\end{pmatrix}
\begin{pmatrix}
D_1^{-1}&O&O\\
O&O&O\\
O&O&O
\end{pmatrix}V^{\rm T},
\end{align*}
 that is,
$N=-D_1^{-1}E_4^{{\rm T}}D_2$ and
$M=-D_2E_2^{{\rm T}}D_1^{-1}$.
Therefore,
$\widehat E\overset{\tiny\mbox{\rm  D\!-}\ast}\leq\widehat F$ from Corollary \ref{tui1}(1).

 ``$\Rightarrow$''
\
 Let $\widehat E\overset{\tiny\mbox{\rm  D\!-}\ast}\leq\widehat F$.
Then $\widehat E\overset{\tiny\mbox{\rm D\!}}\leq\widehat F$
and the forms of $\widehat E$ and $\widehat F$  are as in   Lemma \ref{PEIDSPO-Char-2-Th}.
It follows that
\begin{align*}
EF^{\rm T}
=\left(EF^{\rm T} \right)^{\rm T}
& =U\begin{pmatrix}
D_1^2&O&O\\
O&O&O\\
O&O&O
\end{pmatrix}{ U^{\rm T}},\
F^{\rm T}E= \left(F^{\rm T}E \right)^{\rm T}
={ V}\begin{pmatrix}
D_1^2&O&O\\
O&O&O\\
O&O&O
\end{pmatrix}{ V^{\rm T}}
\\
 EE^\dagger\left(F_0-E_0\right)
 &
 =\left(E^{\rm{ T }}\right)^\dagger E_0^{\rm{ T }}\left(E-F\right)=
U\begin{pmatrix}
O&-D_1^{-1}E_4^{\rm T}D_2&O\\
O&O&O\\
O&O&O
\end{pmatrix}V^{\rm T},
\\
\left(F_0-E_0\right)E^\dagger E
&
=\left(E-F\right)E_0^{\rm{ T }}\left(E^{\rm{ T }}\right)^\dagger=U\begin{pmatrix}
O&O&O\\
-D_2E_2^{\rm T}D_1^{-1}&O&O\\
O&O&O
\end{pmatrix}V^{\rm T}.
\end{align*}
Therefore, (\ref{go}) holds.
\end{proof}

\begin{theorem}
\label{The-nice2}
Let $\widehat E=E+\varepsilon E_0$,
$\widehat F=F+\varepsilon F_0\in\mathbb{D}^{m\times n}$,
and $\widehat E$ and $\widehat F$ have the DMPGIs.
Then $\widehat E\overset{\tiny\mbox{\rm  P\!-}{*}}\leq\widehat F$ if and only if
$\widehat E\overset{\tiny\mbox{\rm D\!}}\leq\widehat F$ and
\begin{align}
\label{go1}
EF^{\rm T}=\left(EF^{\rm T} \right)^{\rm T}, \
F^{\rm T}E=\left(F^{\rm T}E \right)^{\rm T}, \
EE^\dagger\left(F_0-E_0\right)=O,\ \left(F_0-E_0\right)E^\dagger E=O.
\end{align}
\end{theorem}

\begin{proof}
``$\Leftarrow$'' \
Let $\widehat E=E+\varepsilon E_0$,
$\widehat F=F+\varepsilon F_0\in\mathbb{D}^{m\times n}$,
 $\widehat E\overset{\tiny\mbox{\rm D\!}}\leq\widehat F$
 and the forms of  $\widehat E$ and $\widehat F$ be as given in Theorem \ref{fenjie-the}.
Let $EF^{\rm T}=\left(EF^{\rm T} \right)^{\rm T}$ and
$F^{\rm T}E=\left(F^{\rm T}E \right)^{\rm T}$.
Then we have (\ref{fenjieshiAA}) and (\ref{fenjieshiAB}).
Substituting (\ref{fenjieshiAA}) and (\ref{fenjieshiAB})
 into
$EE^\dagger\left(F_0-E_0\right)=O$
and
$\left(F_0-E_0\right)E^\dagger E=O$
gives  $N=O$ and  $M=O$.
According to Corollary \ref{tui1}(2), we get
$\widehat E\overset{\tiny\mbox{\rm  P \!-}*}\leq\widehat F$.

``$\Rightarrow$'' \
Let $\widehat E\overset{\tiny\mbox{\rm  P\!-}{*}}\leq\widehat F$.
Then $\widehat E\overset{\tiny\mbox{\rm D\!}}\leq\widehat F$,
 and the forms of  $\widehat E$ and $\widehat F$ are as given in Lemma \ref{peiDPSO-Char-1-Th}.
 It follows that
 $EF^{\rm T} 
 =\left(EF^{\rm T} \right)^{\rm T}
 =U\left(\begin{smallmatrix}
D_1^2&O&O\\
O&O&O\\
O&O&O
\end{smallmatrix}\right)U^{\rm T}$,
 $F^{\rm T}E
 = \left(F^{\rm T}E \right)^{\rm T}
 =V\left(\begin{smallmatrix}
D_1^2&O&O\\
O&O&O\\
O&O&O
\end{smallmatrix}\right)V^{\rm T}$,
 $EE^\dagger\left(F_0-E_0\right)=O$
 and
$\left(F_0-E_0\right)E^\dagger E=O$.
Therefore, (\ref{go1}) holds.
\end{proof}

 Next, we consider the
 relationships between the Dual-minus partial order and 
 the  D-sharp(G-sharp) partial order. 
 And we apply the Dual-minus partial order to characterizing the D-sharp partial order and G-star partial order.
Firstly, 
a characterization  for the Dual-minus partial order of a dual matrix with an index one is given.
\begin{theorem}
\label{fenjie-2}
Let $\widehat  E=E+\varepsilon E_0, \widehat  F=F+\varepsilon F_0 \in\mathbb{D}_n^{\tiny \mbox{\rm CM}}$.
  Then $\widehat E\overset{\tiny\mbox{\rm  D \!}}\leq\widehat F$
  if and only if
 {\small
\begin{align}
\label{shizi}
&\widehat E=
P\begin{pmatrix}
D_1&O&O\\
O&O&O\\
O&O&O
\end{pmatrix}
P^{-1}
+\varepsilon P\begin{pmatrix}
E_1&E_2&E_3\\
E_4&O&O\\
E_7&O&O\end{pmatrix}P^{-1},
\\
\label{MPO-20230823-2}
&\widehat F=
P\begin{pmatrix}
D_1+RD_2S&RD_2&O\\
D_2S&D_2&O\\
O&O&O
\end{pmatrix}
P^{-1}
+\varepsilon P\begin{pmatrix}
E_1+RM+NS -RF_5S& N+E_2&E_3+RF_6\\
M+ E_4&F_5&F_6\\
E_7+F_8S&F_8&O\end{pmatrix}P^{-1},
\end{align}}
where $D_1\in \mathbb{R}^{r_e\times r_e}$, $D_2\in \mathbb{R}^{(r_f-r_e)\times (r_f-r_e)}$ are invertible,
   and $M, S\in\mathbb{R}^{(r_f-r_e)\times r_e}$, $N, R\in\mathbb{R}^{r_e\times (r_f-r_e)}$ are some suitable matrices.
\end{theorem}

\begin{proof}
 According to Lemma \ref{xiaotuilun}, we have the forms of $E$ and $F$ as  in  (\ref{TUIminus}).
Let  
$\widehat  E=E+\varepsilon E_0\in\mathbb{D}_n^{\tiny \mbox{\rm CM}}$  
and 
$\widehat  F=F+\varepsilon F_0 \in\mathbb{D}_n^{\tiny \mbox{\rm CM}}$.
Then
\begin{align}
\nonumber
&\widehat E=
P\begin{pmatrix}
D_1&O&O\\
O&O&O\\
O&O&O\end{pmatrix}P^{-1}
+\varepsilon P
\begin{pmatrix}
E_1&E_2&E_3\\
E_4&O&O\\
E_7&O&O
\end{pmatrix}P^{-1},
\\
\nonumber
&\widehat F=
P\begin{pmatrix}
D_1+RD_2S&RD_2&O\\
D_2S&D_2&O\\
O&O&O
\end{pmatrix}
P^{-1}
+\varepsilon P\begin{pmatrix}
F_1&F_2&F_3\\
F_4&F_5&F_6\\
F_7&F_8&O\end{pmatrix}P^{-1}.
\end{align}
Similar to the proof process of Theorem \ref{fenjie-the}, 
we can get (\ref{shizi}) and (\ref{MPO-20230823-2}).

On the contrary,
let the forms of $\widehat E$ and $\widehat F$ be as   (\ref{shizi}) and (\ref{MPO-20230823-2}).
Then $\widehat E\overset{\tiny\mbox{\rm  D \!}}\leq\widehat F$ clearly holds 
by Definition \ref{Dual-minusPartialOrder-Def}.
\end{proof}

\begin{corollary}
\label{co-2}
Comparing Theorem  \ref{fenjie-2}
with
Lemma \ref{TIANHEDSPO-Char-2-Th}
and Lemma \ref{tianheDPSPO-Char-1-Th} respectively,
we have

(1) In Theorem \ref{fenjie-2}, giving
 $R=O$, $ S=O$, $M=-D_2E_4D_1^{-1}$ and $N=-D_1^{-1}E_2D_2$,
 it is  obvious that  the form  (\ref{MPO-20230823-2}) of $\widehat F$
 is the same as that in (\ref{pei-DSPO-Char-2}). 
  Therefore, $\widehat E\overset{\tiny\mbox{\rm  D \!-}\#}\leq\widehat F$.

(2) In Theorem  \ref{fenjie-2}, giving
 $R=O$, $S=O$ and $M=O,\ N=O$,
 it is obvious that  the form (\ref{MPO-20230823-2}) of $\widehat F$
 is the same as that in (\ref{DPSPO-Char-1-1}).  Therefore, $\widehat E\overset{\tiny\mbox{\rm  G \!-}\#}\leq\widehat F$.
\end{corollary}

Next, 
we apply the Dual-minus partial order to provide 
 characterizations of  the  D-sharp partial order and the  G-sharp partial order, respectively.

\begin{theorem}
\label{The-nice3}
 Let $\widehat E=E+\varepsilon E_0$,
$\widehat F=F+\varepsilon F_0\in\mathbb{D}_n^{\tiny \mbox{\rm CM}}$.
Then $\widehat E\overset{\tiny\mbox{\rm  D \!-}\#}\leq\widehat F$ if and only if
 $\widehat E\overset{\tiny\mbox{\rm D\!}}\leq\widehat F$ and
  \begin{align}
  \label{go2}
  EF=FE,\   \left(F_0-E_0\right)EE^{\#}=\left(E-F\right)E_0E^{\#},
  \
  EE^{\#}\left(F_0-E_0\right)=E^{\#} E_0\left(E-F\right).
  \end{align}
\end{theorem}

\begin{proof}
``$\Leftarrow$''
\
Let $\widehat E=E+\varepsilon E_0$,
$\widehat F=F+\varepsilon F_0\in\mathbb{D}_n^{\tiny \mbox{\rm CM}}$,
 $\widehat E\overset{\tiny\mbox{\rm D\!}}\leq\widehat F$ and
  the forms of $\widehat E$ and $\widehat F$ be as  in  (\ref{shizi}) and (\ref{MPO-20230823-2}),  respectively.
It follows from $EF=FE$ that $S=O$ and $R=O$.
Then
\begin{align}
\label{fenjieshiAAaa}
&\widehat E=
P\begin{pmatrix}
D_1&O&O\\
O&O&O\\
O&O&O\end{pmatrix}P^{-1}
+\varepsilon P
\begin{pmatrix}
E_1&E_2&E_3\\
E_4&O&O\\
E_7&O&O
\end{pmatrix}P^{-1},
\\
\label{fenjieshiABaa}
&\widehat F=
P\begin{pmatrix}
D_1&O&O\\
O&D_2&O\\
O&O&O
\end{pmatrix}
P^{-1}
+\varepsilon P\begin{pmatrix}
E_1&N+E_2&E_3\\
M+E_4&F_5&F_6\\
E_7&F_8&O\end{pmatrix}P^{-1},
\end{align}
 where $M\in\mathbb{R}^{(r_f-r_e)\times r_e}$ and $N\in\mathbb{R}^{r_e\times (r_f-r_e)}$.

Substituting (\ref{fenjieshiAAaa}) and (\ref{fenjieshiABaa})  into
 $\left(F_0-E_0\right)EE^{\#}=\left(E-F\right)E_0E^{\#}$
 and
 $EE^{\#}\left(F_0-E_0\right)=E^{\#} E_0\left(E-F\right)$  gives
 \begin{align*}
 P
\begin{pmatrix}
O&N&O\\
M&F_5&F_6\\
O&F_8&O\end{pmatrix}
\begin{pmatrix}
 I_{{\rk}\left( E \right)}&O&O\\
O&O&O\\
O&O&O
\end{pmatrix}P^{-1}
&=
P\begin{pmatrix}
O&O&O\\
O&-D_2&O\\
O&O&O
\end{pmatrix}
\begin{pmatrix}
E_1&E_2&E_3\\
E_4&O&O\\
E_7&O&O
\end{pmatrix}
\begin{pmatrix}
D_1^{-1}&O&O\\
O&O&O\\
O&O&O
\end{pmatrix}P^{-1}
\\
P
\begin{pmatrix}
 I_{{\rk}\left( E \right)}&O&O\\
O&O&O\\
O&O&O
\end{pmatrix}
\begin{pmatrix}
O&N&O\\
M&F_5&F_6\\
O&F_8&O\end{pmatrix}P^{-1}
&=
P\begin{pmatrix}
D_1^{-1}&O&O\\
O&O&O\\
O&O&O
\end{pmatrix}
\begin{pmatrix}
E_1&E_2&E_3\\
E_4&O&O\\
E_7&O&O
\end{pmatrix}
\begin{pmatrix}
O&O&O\\
O&-D_2&O\\
O&O&O
\end{pmatrix}
P^{-1}.
\end{align*}
It follows that
 $M=-D_2E_4D_1^{-1}$ and   $N=-D_1^{-1}E_2D_2$.
Therefore, $\widehat E\overset{\tiny\mbox{\rm  D \!-}\#}\leq\widehat F$ from Corollary \ref{co-2}(1).

``$\Rightarrow$'' \
Let $\widehat E\overset{\tiny\mbox{\rm  D\!-}\#}\leq\widehat F$, and
 the forms of $\widehat E$ and $\widehat F$ be as given in  (\ref{TIANHEDSPO-Char-2-Th}).
Then
\begin{align*}
EF
&
=FE
=P\begin{pmatrix}
D_1^2&O&O\\
O&O&O\\
O&O&O
\end{pmatrix}P^{-1},
\\
\left(F_0-E_0\right)EE^{\#}
&
=\left(E-F\right)E_0E^{\#}=
P\begin{pmatrix}
O&O&O\\
-D_2E_4D_1^{-1}&O&O\\
O&O&O
\end{pmatrix}P^{-1},
\\
EE^{\#}\left(F_0-E_0\right)
&
=E^{\#} E_0\left(E-F\right)
 =P\begin{pmatrix}
O&-D_1^{-1}E_2D_2&O\\
O&O&O\\
O&O&O
\end{pmatrix}P^{-1}.
\end{align*}
Therefore, (\ref{go2}) holds.
\end{proof}

\begin{theorem}
\label{The-nice4}
Let $\widehat E=E+\varepsilon E_0$,
$\widehat F=F+\varepsilon F_0\in\mathbb{D}_n^{\tiny \mbox{\rm CM}}$.
Then $\widehat E\overset{\tiny\mbox{\rm  G \!-}\#}\leq\widehat F$ if and only if
 $\widehat E\overset{\tiny\mbox{\rm D\!}}\leq\widehat F$ and
\begin{align}
\label{gogo}
EF=FE, \
\left(F_0-E_0\right)EE^{\#}=O, \
EE^{\#}\left(F_0-E_0\right)=O.
\end{align}
\end{theorem}

\begin{proof}
``$\Leftarrow$''
\
Let $\widehat E=E+\varepsilon E_0$,
$\widehat F=F+\varepsilon F_0\in\mathbb{D}_n^{\tiny \mbox{\rm CM}}$,
 $\widehat E\overset{\tiny\mbox{\rm D\!}}\leq\widehat F$
 and
  the forms of $\widehat E$ and $\widehat F$ be as  in  (\ref{shizi}) and (\ref{MPO-20230823-2}), respectively.
Since $EF=FE$,   we have (\ref{fenjieshiAAaa}) and (\ref{fenjieshiABaa}).
Furthermore,
substituting (\ref{fenjieshiAAaa}) and (\ref{fenjieshiABaa})  into
$EE^{\#}\left(F_0-E_0\right)=O$
and
 $\left(F_0-E_0\right)EE^{\#}=O$
gives $N=O$ and $M=O$.
Therefore, $\widehat E\overset{\tiny\mbox{\rm  G \!-}\#}\leq\widehat F$ from Corollary \ref{co-2}(2).

``$\Rightarrow$''
\
Let $\widehat E\overset{\tiny\mbox{\rm  G \!-}\#}\leq\widehat F$
 and
 the forms of $\widehat E$ and $\widehat F$ be as given in  (\ref{DPSPO-Char-1-1}). 
Then
$EF=FE=
P\left(\begin{smallmatrix}
D_1^2&O&O\\
O&O&O\\
O&O&O
\end{smallmatrix}\right)P^{-1}$,
$\left(F_0-E_0\right)EE^{\#}=O$, and
$EE^{\#}\left(F_0-E_0\right)=O$.
Therefore, (\ref{gogo}) holds.
\end{proof}

\section{Dual-minus sharp Partial Order}\label{Section-4-Dual-minus-sharp-PO}
In \cite{tianhe},
 Wang and Jiang introduce two Dual sharp partial orders,
 namely the D-sharp partial order and the G-sharp partial order, respectively.
In this section,
we introduce a new Dual sharp partial order ---   Dual-minus sharp partial order,
which is different from the existing dual sharp partial orders.
 And the new kind of dual sharp partial orders include the D-sharp(G-sharp) partial orders.

 \begin{definition}
 \label{DM-Sharp-def}
Let
$\widehat E=E+\varepsilon E_0$,
$\widehat F=F+\varepsilon F_0\in\mathbb{D}_n^{\tiny \mbox{\rm CM}}$.
 Write $\widehat E\overset{\tiny\mbox{\rm DM\!-}\#}\leq\widehat F$:
\begin{align}
\label{nice-def-sharp}
E \overset\#\leq F \
\mbox{and the DMPGI of} \ \widehat F-\widehat E \ \mbox{exists}.
\end{align}
We call $\widehat E$ is below $\widehat F$  under the Dual-minus sharp order.
\end{definition}

\begin{theorem}
\label{dms-fenjie}
 Let
$\widehat  E=E+\varepsilon E_0$,
$\widehat  F=F+\varepsilon F_0\in\mathbb{D}_n^{\tiny \mbox{\rm CM}}$,
 ${\rk}\left( E \right)=r_e$ and ${\rk}\left( F \right)=r_f$.
Then  $\widehat E\overset{\tiny\mbox{\rm DM\!-}\#}\leq\widehat F$
if and only if there is
 a nonsingular matrix $P$ such that
\begin{align}
\label{DSPO-Char-2}
\left\{\begin{array}{l}\widehat E
=P\begin{pmatrix}
D_1&O&O\\
O&O&O\\
O&O&O\end{pmatrix}P^{-1}
+\varepsilon P
\begin{pmatrix}
E_1&E_2&E_3\\
E_4&O&O\\
E_7&O&O
\end{pmatrix}P^{-1},\\
\widehat F=
P\begin{pmatrix}D_1&O&O\\
O&D_2&O\\
O&O&O
\end{pmatrix}
P^{-1}
+\varepsilon P\begin{pmatrix}
E_1&F_2&E_3\\
F_4&F_5&F_6\\
E_7&F_8&O\end{pmatrix}P^{-1},
\end{array}\right.
\end{align}
where
$D_1\in \mathbb{R}^{r_e\times r_e}$
and
$D_2\in \mathbb{R}^{(r_f-r_e)\times (r_f-r_e)}$ are nonsingular  matrices.
\end{theorem}

\begin{proof}
Let $\widehat E\overset{\tiny\mbox{\rm DM\!-}\#}\leq\widehat F$.
Then $E \overset\#\leq F$ and the  DMPGIs of $\widehat E$ and $\widehat F$ exist.
It follows from Lemma  \ref{bbbbbb} that
$E$ and $F$ are the forms as in (\ref{3.4}).
Since $\widehat E$ and $\widehat F$ have the DMPGIs,
 we can write
\begin{align}
\label{ye}
\left\{\begin{array}{l}\widehat E
=P\begin{pmatrix}
D_1&O&O\\
O&O&O\\
O&O&O\end{pmatrix}P^{-1}
+\varepsilon P
\begin{pmatrix}
E_1&E_2&E_3\\
E_4&O&O\\
E_7&O&O
\end{pmatrix}P^{-1},\\
\widehat F=
P\begin{pmatrix}D_1&O&O\\
O&D_2&O\\
O&O&O
\end{pmatrix}
P^{-1}
+\varepsilon P\begin{pmatrix}
F_1&F_2&F_3\\
F_4&F_5&F_6\\
F_7&F_8&O\end{pmatrix}P^{-1},
\end{array}\right.
\end{align}
where
$D_1\in \mathbb{R}^{r_e\times r_e}$
and
$D_2\in \mathbb{R}^{(r_f-r_e)\times (r_f-r_e)}$ are nonsingular,
$ E_1, F_1\in \mathbb{R}^{r_e\times r_e}$
and
$F_5\in \mathbb{R}^{(r_f-r_e)\times (r_f-r_e)}$.

Since
 the DMPGI of $\widehat F-\widehat E$ exists,
 ${\rk}\left(\begin{smallmatrix}F_0-E_0&F-E\\F-E&O\end{smallmatrix}\right)=2{\rk}(F-E)$,
that is,
{\small
\begin{align*}
&\rk
\begin{pmatrix}
F_1-E_1&F_2-E_2&F_3-E_3&O&O&O\\
F_4-E_4&F_5&F_6&O&D_2&O\\
F_7-E_7&F_8&O&O&O&O\\
O&O&O&O&O&O\\
O&D_2&O&O&O&O\\
O&O&O&O&O&O
\end{pmatrix}=
\rk
\begin{pmatrix}
F_1-E_1&O&F_3-E_3&O&O&O\\
O&O&O&O&D_2&O\\
F_7-E_7&O&O&O&O&O\\
O&O&O&O&O&O\\
O&D_2&O&O&O&O\\
O&O&O&O&O&O
\end{pmatrix}=2\rk\left(D_2\right).
\end{align*}}
Therefore, $F_1=E_1$, $F_3=E_3$ and $F_7=E_7$.

On the contrary, the conclusion is valid.
\end{proof}

\begin{theorem}
\label{DMS-P}
The Dual-minus sharp order is a partial order.
\end{theorem}

\begin{proof}
 Let $\widehat  E=E+\varepsilon E_0$,
$\widehat  F=F+\varepsilon F_0\in\mathbb{D}_n^{\tiny \mbox{\rm CM}}$,
${\rm \rk}\left( E \right)=r_e$, ${\rm \rk}\left( F \right)=r_f$,
and
 $\widehat E\overset{\tiny\mbox{\rm DM\!-}\#}\leq\widehat F$.

1. Reflexivity: It is evident.

2. Antisymmetry:
Suppose that  $\widehat F\overset{\tiny\mbox{\rm DM\!-}\#}\leq\widehat E$.
Since $\widehat E\overset{\tiny\mbox{\rm DM\!-}\#}\leq\widehat F$
  and $\widehat F\overset{\tiny\mbox{\rm DM\!-}\#}\leq\widehat E$,
we have that the DMPGI of $\widehat F-\widehat E$ exists, $E \overset\#\leq F$
 and $F \overset\#\leq E$.
Since
$E \overset\#\leq F$
 and $F \overset\#\leq E$, we have $E=F$.
It follows that
\begin{align*}
{\rk}\begin{pmatrix}F_0-E_0&F-E\\F-E&O\end{pmatrix}
=
{\rk}\begin{pmatrix}F_0-E_0&O\\O&O\end{pmatrix}
=
2{\rk}(F-E)=0.
\end{align*}
Therefore,  $E_0=F_0$. So $\widehat E =\widehat F$.

3. Transitivity: \
Suppose that
$\widehat  G=G+\varepsilon G_0\in\mathbb{D}_n^{\tiny \mbox{\rm CM}}$,
${\rm \rk}\left( G \right)=r_g$,
the $\widehat G$ has the DMPGI and
 $\widehat F\overset{\tiny\mbox{\rm DM\!-}\#}\leq\widehat G$.
 Let $\widehat E\overset{\tiny\mbox{\rm DM\!-}\#}\leq\widehat F$
  and $\widehat F\overset{\tiny\mbox{\rm DM\!-}\#}\leq\widehat G$.
Then  $E \overset\#\leq F$  and $F \overset\#\leq G$.
By applying Lemma \ref{bbbbbb} and Theorem \ref{dms-fenjie},
we get
\begin{align*}
 \widehat E
&=
P\begin{pmatrix}
D_1&O&O&O\\
O&O&O&O\\
O&O&O&O\\
O&O&O&O
\end{pmatrix}P^{-1}
+\varepsilon P
\begin{pmatrix}
E_1&E_2&E_{3}&E_{4}\\
E_5&O&O&O\\
E_{9}&O&O&O\\
E_{13}&O&O&O
\end{pmatrix}P^{-1},\\
\widehat F
&=
P\begin{pmatrix}
D_1&O&O&O\\
O&D_2&O&O\\
O&O&O&O\\
O&O&O&O
\end{pmatrix}
P^{-1}
+\varepsilon P\begin{pmatrix}
E_1&F_2&E_{3}&E_{4}\\
F_5&F_6&F_{7}&F_{8}\\
E_{9}&F_{10}&O&O\\
E_{13}&F_{14}&O&O
\end{pmatrix}P^{-1},
\\
\widehat G
&=
P\begin{pmatrix}
D_1&O&O&O\\
O&D_2&O&O\\
O&O&D_3&O\\
O&O&O&O
\end{pmatrix}
P^{-1}
+\varepsilon P\begin{pmatrix}
E_1&F_2&G_{3}&E_{4}\\
F_5&F_6&G_{7}&F_{8}\\
G_{9}&G_{10}&G_{11}&G_{12}\\
E_{13}&F_{14}&G_{15}&O
\end{pmatrix}P^{-1},
\end{align*}
where
$D_1\in \mathbb{R}^{r_e\times r_e}$,
$D_2\in \mathbb{R}^{(r_f-r_e)\times (r_f-r_e)}$
and
$D_3\in \mathbb{R}^{(r_g-r_f)\times (r_g-r_f)}$ are nonsingular  matrices,
and
$E_1\in \mathbb{R}^{r_e\times r_e}$,
$F_6\in \mathbb{R}^{(r_f-r_e)\times (r_f-r_e)}$
and
$G_{11}\in \mathbb{R}^{(r_g-r_f)\times (r_g-r_f)}$.
It is easy to check that  $E \overset\#\leq G$ and
the DMPGI of $\widehat G-\widehat E$ exists, that is,
$\widehat E\overset{\tiny\mbox{\rm DM\!-}\#}\leq\widehat G$.
\end{proof}

Next, we consider the relationship between the Dual-minus sharp  partial order and the Dual-minus partial order.

\begin{theorem}
\label{20231209-Th4-3}
Let
$\widehat  E $,
$\widehat  F \in\mathbb{D}_n^{\tiny \mbox{\rm CM}}$.
If $\widehat E\overset{\tiny\mbox{\rm DM\!-}\#}\leq\widehat F$, then
$\widehat E\overset{\tiny\mbox{\rm D\!}}\leq\widehat F$.
\end{theorem}

\begin{proof}
Let
$\widehat  E=E+\varepsilon E_0$,
$\widehat  F=F+\varepsilon F_0\in\mathbb{D}_n^{\tiny \mbox{\rm CM}}$,
  ${\rk}\left( E \right)=r_e$ and ${\rk}\left( F \right)=r_f$.
 Since  $\widehat E\overset{\tiny\mbox{\rm DM\!-}\#}\leq\widehat F$,  
   $\widehat E$ and  $\widehat F$ have the forms  (\ref{DSPO-Char-2}).
And then, $E\leq F$ and
\begin{align}
\nonumber
\widehat F-\widehat E
 =\left(F-E\right)+\varepsilon \left(F_0-E_0\right)=
P\begin{pmatrix}
O&O&O\\
O&D_2&O\\
O&O&O
\end{pmatrix}
P^{-1}
+\varepsilon P\begin{pmatrix}
O&F_2-E_2&O\\
F_4-E_4&F_5&F_6\\
O&F_8&O\end{pmatrix}P^{-1}.
\end{align}
Therefore,
we get that $E\leq F$ and the DMPGI of $\widehat F-\widehat E$ exists,
that is, $\widehat E\overset{\tiny\mbox{\rm D\!}}\leq\widehat F$.
\end{proof}

The converse of Theorem \ref{20231209-Th4-3} is not true.

\begin{example}
\label{first}
Let
$$
\widehat E=E+\varepsilon E_0
=
 \begin{pmatrix}
1&0&0\\0&0&0\\0&0&0\end{pmatrix}
+\varepsilon
\begin{pmatrix}
1&1&1\\1&0&0\\1&0&0
\end{pmatrix}, \
\widehat F=F+\varepsilon F_0
=
 \begin{pmatrix}
2&1&0\\1&1&0\\0&0&0
\end{pmatrix}
+\varepsilon  \begin{pmatrix}
4&3&2\\3&1&1\\2&1&0
\end{pmatrix}.
$$
Then the  DMPGIs of $\widehat E$ and $\widehat F$ exist
and
$$\widehat F-\widehat E=
\begin{pmatrix}
1&1&0\\1&1&0\\0&0&0
\end{pmatrix}
+\varepsilon  \begin{pmatrix}
3&2&1\\2&1&1\\1&1&0
\end{pmatrix},  \
{\rk}\left(\begin{matrix}
F_0-E_0  &F-E\\ F-E  &O
\end{matrix}\right)
 =2{\rk}\left( F-E \right)=2.
$$
And since
  $\rk\left(F-E\right)=1=\rk(F)-\rk(E)$,
  we get that  $\widehat E \overset{\tiny\mbox{\rm D\!}}\leq \widehat F$.

Because
$$E^\#E=\begin{pmatrix}
1&0&0\\0&0&0\\0&0&0
\end{pmatrix}, \
E^\#F=\begin{pmatrix}
2&1&0\\0&0&0\\0&0&0
\end{pmatrix},$$
we have $E^\#E\neq E^\#F$.
 That means $E \overset\#\leq F$ does not hold.
It is evident that $\widehat E$ cannot be below  $\widehat F$ under the Dual-minus sharp partial order.
\end{example}

In the following Theorem \ref{nice---6},
based on the Dual-minus partial order, we give a  characterization of the Dual-minus sharp partial order.
\begin{theorem}
\label{nice---6}
Let $\widehat E=E+\varepsilon E_0$,
$\widehat F=F+\varepsilon F_0\in\mathbb{D}_n^{\tiny \mbox{\rm CM}}$.
Then $\widehat E\overset{\tiny\mbox{\rm DM\!-}\#}\leq\widehat F$ if and only if
$\widehat E\overset{\tiny\mbox{\rm D\!}}\leq\widehat F$ and
 \begin{align}
\label{20231010-1}
EF=FE.
\end{align}
\end{theorem}

\begin{proof}
``$\Leftarrow$''
\
Let $\widehat E=E+\varepsilon E_0$,
$\widehat F=F+\varepsilon F_0\in\mathbb{D}_n^{\tiny \mbox{\rm CM}}$,  and
$\widehat E\overset{\tiny\mbox{\rm D\!}}\leq\widehat F$.
Then $E\leq F$ and
 the DMPGI of $\widehat F-\widehat E$ exists.
And since $EF=FE$, by applying Lemma \ref{bbbbbb}(2),
 we get $E\overset\#\leq F$.
Therefore,
$\widehat E\overset{\tiny\mbox{\rm DM\!-}\#}\leq\widehat F$.

``$\Rightarrow$'' \
Let $\widehat E\overset{\tiny\mbox{\rm DM\!-}\#}\leq\widehat F$.
We get $E \leq F$ and the DMPGI  of $\widehat F-\widehat E$  exists
by applying Definition \ref{DM-Sharp-def}.
Therefore, $\widehat E\overset{\tiny\mbox{\rm D\!}}\leq\widehat F$.
Since   $\widehat E\overset{\tiny\mbox{\rm DM\!-}\#}\leq\widehat F$,
 we get $EF=FE$
 from Theorem \ref{dms-fenjie}.
\end{proof}

Next, we consider the relationship between the Dual-minus sharp  partial order and the D-sharp partial  order.

\begin{theorem}
\label{D-DM-Sharp-Th}
 Let
$\widehat  E=E+\varepsilon E_0$,
$\widehat  F=F+\varepsilon F_0\in\mathbb{D}_n^{\tiny \mbox{\rm CM}}$.
If $\widehat E\overset{\tiny\mbox{\rm D\!-}\#}\leq\widehat F$, then
$\widehat E\overset{\tiny\mbox{\rm DM\!-}\#}\leq\widehat F$.
\end{theorem}
\begin{proof}
 Let $\widehat  E$, $\widehat  F\in\mathbb{D}_n^{\tiny \mbox{\rm CM}}$,
${\rk}\left( E \right)=r_e$, ${\rk}\left( F \right)=r_f$ and
 $\widehat E\overset{\tiny\mbox{\rm D\!-}\#}\leq\widehat F$.
  Then we obtain  that
  $E\overset\#\leq F$
  and
 the forms of $\widehat E$ and  $\widehat F$ are as in (\ref{pei-DSPO-Char-2}) by Lemma \ref{TIANHEDSPO-Char-2-Th}.
It follows from
\begin{align*}
\widehat F-\widehat E
 =\left(F-E\right)+\varepsilon \left(F_0-E_0\right)=
P\begin{pmatrix}
O&O&O\\
O&D_2&O\\
O&O&O
\end{pmatrix}
P^{-1}
+\varepsilon P\begin{pmatrix}
O&-D_1^{-1}E_2D_2&O\\
-D_2E_4D_1^{-1}&F_5&F_6\\
O&F_8&O\end{pmatrix}P^{-1}
\end{align*}
that the DMPGI of $\widehat F-\widehat E$ exists.
 Therefore,  according to Definition
\ref{DM-Sharp-def},
we have $\widehat E\overset{\tiny\mbox{\rm DM\!-}\#}\leq\widehat F$.
\end{proof}

The converse of Theorem \ref{D-DM-Sharp-Th} is not true.

\begin{example}
Let
\begin{align}
 \nonumber
\widehat E=E+\varepsilon E_0
=
 \begin{pmatrix}
1&0&0\\0&0&0\\0&0&0\end{pmatrix}
+\varepsilon
\begin{pmatrix}
1&4&7\\2&0&0\\3&0&0
\end{pmatrix}, \
\widehat F=F+\varepsilon F_0
=
 \begin{pmatrix}
1&0&0\\0&1&0\\0&0&0
\end{pmatrix}
+\varepsilon  \begin{pmatrix}
1&4&7\\2&-1&-2\\3&-3&0
\end{pmatrix}
\end{align}
then
$$\widehat F-\widehat E=
\begin{pmatrix}
0&0&0\\0&1&0\\0&0&0
\end{pmatrix}
+\varepsilon  \begin{pmatrix}
0&0&0\\0&-1&-2\\0&-3&0
\end{pmatrix}, \
{\rk}\left(\begin{matrix}
F_0-E_0  &F-E\\ F-E  &O
\end{matrix}\right)
=2{\rk}\left( F-E \right)=2.$$
 Therefore, the  DMPGI  of $\widehat F-\widehat E$  exists.
Since $E^2=E F=  FE $,  we have $E\overset\#\leq F$.
It follows that
 $\widehat E\overset{\tiny\mbox{\rm DM\!-}\#}\leq\widehat F$.

 And since
$$EF_0+E_0F
=\begin{pmatrix}
2&8&7\\2&0&0\\3&0&0
\end{pmatrix}, \
E_0E+EE_0=
\begin{pmatrix}
2&4&7\\2&0&0\\3&0&0
\end{pmatrix},$$
then
$EF_0+E_0F \neq E_0E+EE_0$.
It is evident that $\widehat E$ cannot be below  $\widehat F$ 
under the  D-sharp partial order by applying Lemma \ref{TIANHEDSPO-Char-2-Th}(2).
\end{example}

In the following Theorem \ref{nice---4},
based on the Dual-minus sharp partial order, we provide a  characterization of the  D-sharp   partial order.

\begin{theorem}
\label{nice---4}
Let $\widehat E=E+\varepsilon E_0$,
$\widehat F=F+\varepsilon F_0\in\mathbb{D}_n^{\tiny \mbox{\rm CM}}$.
 Then
 $\widehat E\overset{\tiny\mbox{\rm  D \!-}\#}\leq\widehat F$
  if and only if
 $\widehat E\overset{\tiny\mbox{\rm DM\!-}\#}\leq\widehat F$ and
\begin{align}
\label{perfect-1}
\left(F_0-E_0\right)EE^{\#}=\left(E-F\right)E_0E^{\#},\
EE^{\#}\left(F_0-E_0\right)=E^{\#} E_0\left(E-F\right).
\end{align}
\end{theorem}

\begin{proof}
``$\Leftarrow$'' \
Let $\widehat E=E+\varepsilon E_0$,
$\widehat F=F+\varepsilon F_0\in\mathbb{D}_n^{\tiny \mbox{\rm CM}}$,
and
 $\widehat E\overset{\tiny\mbox{\rm DM\!-}\#}\leq\widehat F$.
Then the forms of $\widehat E$ and $\widehat F$ are as in
 (\ref{DSPO-Char-2}).
It is easy to check that
\begin{align}
\label{20231008-1}
\left(F_0-E_0\right)EE^{\#}
&=P
\begin{pmatrix}
O&F_2-E_2&O\\
F_4-E_4&F_5&F_6\\
O&F_8&O\end{pmatrix}
\begin{pmatrix}
I_{{\rk}\left( E \right)}&O&O\\
O&O&O\\
O&O&O
\end{pmatrix}P^{-1}
=P\begin{pmatrix}
O&O&O\\
F_4-E_4&O&O\\
O&O&O
\end{pmatrix}P^{-1},
\\
\nonumber
\left(E-F\right)E_0E^{\#}
&=
P\begin{pmatrix}
O&O&O\\
O&{ -D_2}&O\\
O&O&O
\end{pmatrix}
\begin{pmatrix}
E_1&E_2&E_3\\
E_4&O&O\\
E_7&O&O
\end{pmatrix}
\begin{pmatrix}
D_1^{-1}&O&O\\
O&O&O\\
O&O&O
\end{pmatrix}P^{-1}=P\begin{pmatrix}
O&O&O\\
-D_2E_4D_1^{-1}&O&O\\
O&O&O
\end{pmatrix}P^{-1}.
\end{align}
It follows from $\left(F_0-E_0\right)EE^{\#}=\left(E-F\right)E_0E^{\#}$
that   $F_4=E_4-D_2E_4D_1^{-1}$.

Since
\begin{align}
\label{20231008-2}
EE^{\#}\left(F_0-E_0\right)
&
=
P\begin{pmatrix}
I_{{\rk}\left( E \right)}&O&O\\
O&O&O\\
O&O&O
\end{pmatrix}
\begin{pmatrix}
O&F_2-E_2&O\\
F_4-E_4&F_5&F_6\\
O&F_8&O\end{pmatrix}P^{-1}
=P\begin{pmatrix}
O&F_2-E_2&O\\
O&O&O\\
O&O&O
\end{pmatrix}P^{-1},
\\
\nonumber
E^{\#} E_0\left(E-F\right)
&
=
P\begin{pmatrix}
D_1^{-1}&O&O\\
O&O&O\\
O&O&O
\end{pmatrix}
\begin{pmatrix}
E_1&E_2&E_3\\
E_4&O&O\\
E_7&O&O
\end{pmatrix}
\begin{pmatrix}
O&O&O\\
O&-D_2&O\\
O&O&O
\end{pmatrix}
P^{-1}
=
P\begin{pmatrix}
O&-D_1^{-1}E_2D_2&O\\
O&O&O\\
O&O&O
\end{pmatrix}P^{-1},
\end{align}
applying $EE^{\#}\left(F_0-E_0\right)=E^{\#} E_0\left(E-F\right)$
we have  $F_2=E_2-D_1^{-1}E_2D_2$.

Substituting $F_4=E_4-D_2E_4D_1^{-1}$ and $F_2=E_2-D_1^{-1}E_2D_2$
into (\ref{DSPO-Char-2}) gives (\ref{pei-DSPO-Char-2}).
Therefore,
$\widehat E\overset{\tiny\mbox{\rm  D \!-}\#}\leq\widehat F$ holds.

``$\Rightarrow$''
\
Let $\widehat E\overset{\tiny\mbox{\rm  D \!-}\#}\leq\widehat F$.
It follows from Theorem \ref{D-DM-Sharp-Th}
that
$\widehat E\overset{\tiny\mbox{\rm DM\!-}\#}\leq\widehat F$
and
  the the forms of
$\widehat E$ and $\widehat F$ are as in (\ref{pei-DSPO-Char-2}).
Since
\begin{align*}
\left(F_0-E_0\right)EE^{\#}
&
=
\left(E-F\right)E_0E^{\#}=
P\begin{pmatrix}
O&O&O\\
-D_2E_4D_1^{-1}&O&O\\
O&O&O
\end{pmatrix}P^{-1},
\\
EE^{\#}\left(F_0-E_0\right)
&
=
E^{\#} E_0\left(E-F\right)=
P\begin{pmatrix}
O&-D_1^{-1}E_2D_2&O\\
O&O&O\\
O&O&O
\end{pmatrix}P^{-1},
\end{align*}
then
 (\ref{perfect-1}) holds.
\end{proof}

Next,
we consider the relationship between the Dual-minus sharp order and the G-sharp partial  order.

\begin{theorem}
\label{G-DM-Sharp-Th}
Let
$\widehat  E=E+\varepsilon E_0$,
$\widehat  F=F+\varepsilon F_0\in\mathbb{D}_n^{\tiny \mbox{\rm CM}}$.
If $\widehat E\overset{\tiny\mbox{\rm G\!-}\#}\leq\widehat F$, then
$\widehat E\overset{\tiny\mbox{\rm DM\!-}\#}\leq\widehat F$.
\end{theorem}

\begin{proof}
 Let $\widehat  E$, $\widehat  F\in\mathbb{D}_n^{\tiny \mbox{\rm CM}}$,
${\rk}\left( E \right)=r_e$, ${\rk}\left( F \right)=r_f$ 
and 
$\widehat E\overset{\tiny\mbox{\rm G\!-}\#}\leq\widehat F$.
Then  $E\overset\#\leq F$
  and
 the forms of $\widehat E$ and  $\widehat F$ are as in (\ref{DPSPO-Char-1-1})
  by Lemma \ref{tianheDPSPO-Char-1-Th}.
It follows from
\begin{align}
\nonumber
\widehat F-\widehat E
 =\left(F-E\right)+\varepsilon \left(F_0-E_0\right)=
P\begin{pmatrix}
O&O&O\\
O&D_2&O\\
O&O&O
\end{pmatrix}
P^{-1}
+\varepsilon P\begin{pmatrix}
O&O&O\\
O&F_5&F_6\\
O&F_8&O\end{pmatrix}P^{-1}
\end{align}
that the DMPGI of $\widehat F-\widehat E$ exists.
 Therefore,  according to Definition
\ref{DM-Sharp-def},
we have $\widehat E\overset{\tiny\mbox{\rm DM\!-}\#}\leq\widehat F$.
\end{proof}

The converse of Theorem \ref{G-DM-Sharp-Th} is not true.

\begin{example}
Let
\begin{align}
 \nonumber
\widehat E=E+\varepsilon E_0
=
 \begin{pmatrix}
1&0&0\\0&0&0\\0&0&0\end{pmatrix}
+\varepsilon
\begin{pmatrix}
1&2&3\\2&0&0\\3&0&0
\end{pmatrix}, \
\widehat F=F+\varepsilon F_0
=
 \begin{pmatrix}
1&0&0\\0&1&0\\0&0&0
\end{pmatrix}
+\varepsilon  \begin{pmatrix}
1&6&3\\6&0&0\\3&0&0
\end{pmatrix}.
\end{align}
Then
\begin{align}
\widehat F-\widehat E=
\begin{pmatrix}
0&0&0\\0&1&0\\0&0&0
\end{pmatrix}
+\varepsilon  \begin{pmatrix}
0&4&0\\4&0&0\\0&0&0
\end{pmatrix},\
{\rk}\left(\begin{matrix}
F_0-E_0  &F-E\\ F-E  &O
\end{matrix}\right)
=2=2{\rk}\left( F-E \right) .
\end{align}
 Therefore, the  DMPGI  of $\widehat F-\widehat E$  exists.
Since $E^2=E F=  FE $,  we have $E\overset\#\leq F$.
It follows that
 $\widehat E\overset{\tiny\mbox{\rm DM\!-}\#}\leq\widehat F$.

 And since $$E_0E
=\begin{pmatrix}
1&0&0\\2&0&0\\3&0&0
\end{pmatrix},\
F_0E=
\begin{pmatrix}
1&0&0\\6&0&0\\3&0&0
\end{pmatrix},$$
$E_0E \neq F_0E$.
It is evident that $\widehat E$ cannot  be below $\widehat F$ 
under the  G-sharp partial order by Lemma \ref{tianheDPSPO-Char-1-Th}(3).
\end{example}

In the following Theorem \ref{nice---5},
 based on  the Dual-minus sharp partial order, we provide  a  characterization of the G-sharp  partial order.
\begin{theorem}
\label{nice---5}
Let $\widehat E=E+\varepsilon E_0$,
$\widehat F=F+\varepsilon F_0\in\mathbb{D}_n^{\tiny \mbox{\rm CM}}$.
Then $\widehat E\overset{\tiny\mbox{\rm  G \!-}\#}\leq\widehat F$ if and only if
 $\widehat E\overset{\tiny\mbox{\rm DM\!-}\#}\leq\widehat F$
 and
 \begin{align}
\label{perfect-4}
\left(F_0-E_0\right)EE^{\#}=O, \
EE^{\#}\left(F_0-E_0\right)=O.
\end{align}
\end{theorem}

\begin{proof}
``$\Leftarrow$'' \
Let $\widehat E=E+\varepsilon E_0$,
$\widehat F=F+\varepsilon F_0\in\mathbb{D}_n^{\tiny \mbox{\rm CM}}$,
and
 $\widehat E\overset{\tiny\mbox{\rm DM\!-}\#}\leq\widehat F$.
 Then the forms of $\widehat E$ and $\widehat F$ are as in
 (\ref{DSPO-Char-2}).
Since
 $EE^{\#}\left(F_0-E_0\right)=O$
 and
 $\left(F_0-E_0\right)EE^{\#}=O$, 
we get  $F_4=E_4$ and $F_2=E_2$.
Therefore,
$\widehat E\overset{\tiny\mbox{\rm  G \!-}\#}\leq\widehat F$ holds  by Lemma \ref{tianheDPSPO-Char-1-Th}(4).

``$\Rightarrow$'' \
Let $\widehat E\overset{\tiny\mbox{\rm  D \!-}\#}\leq\widehat F$.
It follows from Theorem \ref{G-DM-Sharp-Th}
that
$\widehat E\overset{\tiny\mbox{\rm DM\!-}\#}\leq\widehat F$
and
  the the forms of
$\widehat E$ and $\widehat F$ are as in (\ref{DPSPO-Char-1-1}).
It is easy to check that
 $\left(F_0-E_0\right)EE^{\#}=O$
 and
$EE^{\#}\left(F_0-E_0\right)=O$.
Therefore, (\ref{perfect-4}) holds.
\end{proof}

\section{Dual-minus star Partial Order}
\label{Section-5-Dual-minus-star-PO}
In Section \ref{Section-4-Dual-minus-sharp-PO},
we introduce the Dual-minus sharp partial order
by applying the sharp partial order and the DMPGI.
In this section, we introduce  a new Dual star partial order ---    Dual-minus star partial order,
which is different from the D-star partial order and the P-star partial order.
The proofs in this section  are similar to that in Section \ref{Section-4-Dual-minus-sharp-PO}.

\begin{definition}
\label{def-DMS}
Let
$\widehat E=E+\varepsilon E_0$,
$\widehat F=F+\varepsilon F_0\in\mathbb{D}^{m\times n}$,
and $\widehat  E$ and $\widehat  F$ have the DMPGIs.
 Denote
$\widehat E\overset{\tiny\mbox{\rm DM\!-}{*}}\leq\widehat F$:
\begin{align}
\label{nice-def}
E \overset{*}\leq F \
\mbox{and the DMPGI of} \ \widehat F-\widehat E \ \mbox{exists}.
\end{align}
We call $\widehat E$ is below $\widehat F$  under the
 Dual-minus star order.
\end{definition}

 \begin{theorem}
\label{dmstar-fenjieshi}
Let
$\widehat  E=E+\varepsilon E_0$,
$\widehat  F=F+\varepsilon F_0\in\mathbb{D}^{m\times n}$,
${\rk}\left( E \right)=r_e$, ${\rk}\left( F \right)=r_f$,
and $\widehat  E$ and $\widehat  F$ have the DMPGIs.
Then  $\widehat E\overset{\tiny\mbox{\rm DM\!-}{*}}\leq\widehat F$
 if and only if there are
 orthogonal matrices $U$ and $V$ such that
{\small
\begin{align}
\label{DSPO-Char-2star}
\left\{\begin{array}{l}\widehat E
=U\begin{pmatrix}
D_1&O&O\\
O&O&O\\
O&O&O\end{pmatrix}V^{\rm {T}}
+\varepsilon U
\begin{pmatrix}
E_1&E_2&E_3\\
E_4&O&O\\
E_7&O&O
\end{pmatrix}V^{\rm {T}},\\
\widehat F=
U\begin{pmatrix}
D_1&O&O\\
O&D_2&O\\
O&O&O
\end{pmatrix}
V^{\rm {T}}
+\varepsilon U\begin{pmatrix}
E_1&F_2&E_3\\
F_4&F_5&F_6\\
E_7&F_8&O\end{pmatrix}V^{\rm {T}},
\end{array}\right.
\end{align}}
where
$D_1 \in \mathbb{R}^{r_e\times r_e}$
and
$D_2 \in \mathbb{R}^{(r_f-r_e)\times (r_f-r_e)}$
  are invertible.
\end{theorem}

\begin{proof}
Let $\widehat E\overset{\tiny\mbox{\rm DM\!-}{*}}\leq\widehat F$.
Then  $E \overset{*}\leq F$.
According to Lemma \ref{peiRStar-Def},
we obtain that  $E$ and $F$ are the forms of in (\ref{2.1}).
Because $\widehat  E$ and $\widehat  F$ have the DMPGIs,
we can  denote
\begin{align}
\label{SHOUJI}
\left\{\begin{array}{l}\widehat E
=U\begin{pmatrix}
D_1&O&O\\
O&O&O\\
O&O&O\end{pmatrix}V^{\rm {T}}
+\varepsilon U
\begin{pmatrix}
E_1&E_2&E_3\\
E_4&O&O\\
E_7&O&O
\end{pmatrix}V^{\rm {T}},\\
\widehat F=
U\begin{pmatrix}D_1&O&O\\
O&D_2&O\\
O&O&O
\end{pmatrix}
V^{\rm {T}}
+\varepsilon U\begin{pmatrix}
F_1&F_2&F_3\\
F_4&F_5&F_6\\
F_7&F_8&O\end{pmatrix}V^{\rm {T}},
\end{array}\right.
\end{align}
where
$D_1 \in \mathbb{R}^{r_e\times r_e}$
and
$D_2 \in \mathbb{R}^{(r_f-r_e)\times (r_f-r_e)}$
are invertible,
$ E_1, F_1\in \mathbb{R}^{r_e\times r_e}$
and
$F_5\in \mathbb{R}^{(r_f-r_e)\times (r_f-r_e)}$.

And since
the DMPGI of $\widehat F-\widehat E$ exists,
${\rk}\left(\begin{smallmatrix}
F_0-E_0&F-E
\\F-E&O\end{smallmatrix}\right)
=2~{\rk}(F-E)$,
 that is,
\small{
\begin{align*}
&\rk
\begin{pmatrix}
F_1-E_1&F_2-E_2&F_3-E_3&O&O&O\\
F_4-E_4&F_5&F_6&O&D_2&O\\
F_7-E_7&F_8&O&O&O&O\\
O&O&O&O&O&O\\
O&D_2&O&O&O&O\\
O&O&O&O&O&O
\end{pmatrix}=
\rk
\begin{pmatrix}
F_1-E_1&O&F_3-E_3&O&O&O\\
O&O&O&O&D_2&O\\
F_7-E_7&O&O&O&O&O\\
O&O&O&O&O&O\\
O&D_2&O&O&O&O\\
O&O&O&O&O&O
\end{pmatrix}=2\rk\left(D_2\right).
\end{align*}}
It follows that $F_1=E_1$, $F_3=E_3$ and $F_7=E_7$.
Therefore,   (\ref{DSPO-Char-2star}) holds.

On the contrary, the conclusion is valid.
\end{proof}

\begin{theorem}
The Dual-minus star order is a partial order.
\end{theorem}

\begin{proof}
 Let
$\widehat E=E+\varepsilon E_0$,
$\widehat F=F+\varepsilon F_0\in\mathbb{D}^{m\times n}$,
 ${ \rk}\left( E \right)=r_e$, ${ \rk}\left( F \right)=r_f$,
  $\widehat  E$ and $\widehat  F$ have the DMPGIs,
and
$\widehat E\overset{\tiny\mbox{\rm DM\!-}{*}}\leq\widehat F$.
   Next,
  we prove that Dual-minus star order is reflexive,
  anti-symmetric and transitive.

1. Reflexivity: \ It is trivial.

2. Antisymmetry: \
Suppose that  $\widehat F\overset{\tiny\mbox{\rm DM\!-}*}\leq\widehat E$.
Since $\widehat E\overset{\tiny\mbox{\rm DM\!-}*}\leq\widehat F$
  and $\widehat F\overset{\tiny\mbox{\rm DM\!-}*}\leq\widehat E$,
we get that the DMPGI of $\widehat F-\widehat E$ exists, $E \overset{*}\leq F$
 and $F \overset{*}\leq E$.
Since
$E \overset{*}\leq F$
 and $F \overset{*}\leq E$, we have $E=F$.
It follows that
\begin{align*}
{\rk}\begin{pmatrix}F_0-E_0&F-E\\F-E&O\end{pmatrix}
=
{\rk}\begin{pmatrix}F_0-E_0&O\\O&O\end{pmatrix}
=
2{\rk}(F-E)=O.
\end{align*}
Therefore,  $E_0=F_0$. So $\widehat E =\widehat F$.

3. Transitivity:
\
Suppose that
 $\widehat  G=G+\varepsilon G_0\in\mathbb{D}^{m\times n}$,
${\rm \rk}\left( G \right)=r_g$,
 $\widehat G$  has the  DMPGI and
 $\widehat F\overset{\tiny\mbox{\rm DM\!-}{*}}\leq\widehat G$.
Since $\widehat E\overset{\tiny\mbox{\rm DM\!-}{*}}\leq\widehat F$
  and $\widehat F\overset{\tiny\mbox{\rm DM\!-}{*}}\leq\widehat G$,
then  $E \overset{*}\leq F$  and $F \overset{*}\leq G$.
 It is easy to check that  $E \overset{*}\leq G$ and
the DMPGI of $\widehat G-\widehat E$ exists
 by applying
 Theorem \ref{dmstar-fenjieshi}. Therefore,
$\widehat E\overset{\tiny\mbox{\rm DM\!-}{*}}\leq\widehat G$.
\end{proof}

Next, we study the relationship between the Dual-minus star order and the Dual-minus order.

\begin{theorem}
Let
$\widehat  E=E+\varepsilon E_0$,
$\widehat  F=F+\varepsilon F_0\in\mathbb{D}^{m\times n}$,
and $\widehat  E$ and $\widehat  F$ have the DMPGIs.
If $\widehat E\overset{\tiny\mbox{\rm DM\!-}{*}}\leq\widehat F$, then
$\widehat E\overset{\tiny\mbox{\rm D\!}}\leq\widehat F$.
\end{theorem}

\begin{proof}
Let
$\widehat  E=E+\varepsilon E_0$,
$\widehat  F=F+\varepsilon F_0\in\mathbb{D}^{m\times n}$,
 $\widehat  E$ and $\widehat  F$ have the DMPGIs,
 and $\widehat E\overset{\tiny\mbox{\rm DM\!-}{*}}\leq\widehat F$.
Then according to Theorem \ref{dmstar-fenjieshi},
 it is obvious that $\widehat E\overset{\tiny\mbox{\rm D\!}}\leq\widehat F$.
\end{proof}

\begin{theorem}
\label{ok---6}
Let $\widehat E=E+\varepsilon E_0$,
$\widehat F=F+\varepsilon F_0\in\mathbb{D}^{m\times n}$. Then
$\widehat E\overset{\tiny\mbox{\rm DM\!-}*}\leq\widehat F$  if and only if
$\widehat E\overset{\tiny\mbox{\rm D\!}}\leq\widehat F$ and
\begin{align}
\label{perfect-6}
 EF^{\rm T}=\left(EF^{\rm T} \right)^{\rm T}, \
F^{\rm T}E=\left(F^{\rm T}E \right)^{\rm T}.
\end{align}
\end{theorem}

\begin{proof}
``$\Leftarrow$'' \
Let $\widehat E=E+\varepsilon E_0$,
$\widehat F=F+\varepsilon F_0\in\mathbb{D}^{m\times n}$
and
$\widehat E\overset{\tiny\mbox{\rm D\!}}\leq\widehat F$.
Then $E\leq F$ and
 the DMPGI of $\widehat F-\widehat E$ exists.
And because $ EF^{\rm T}=\left(EF^{\rm T} \right)^{\rm T}$
and
$F^{\rm T}E=\left(F^{\rm T}E \right)^{\rm T}$,
by applying Lemma \ref{peiRStar-Def}(2),
 we get $E\overset{*}\leq F$.
Therefore,
$\widehat E\overset{\tiny\mbox{\rm DM\!-}*}\leq\widehat F$.

``$\Rightarrow$'' \
Let $\widehat E\overset{\tiny\mbox{\rm DM\!-}*}\leq\widehat F$.
It follows from Definition \ref{nice-def}
that $E \leq F$ and the DMPGI  of $\widehat F-\widehat E$  exists.
Therefore, $\widehat E\overset{\tiny\mbox{\rm D\!}}\leq\widehat F$.
Since  $\widehat E\overset{\tiny\mbox{\rm DM\!-}*}\leq\widehat F$,
we get $ EF^{\rm T}=\left(EF^{\rm T} \right)^{\rm T}$
and
$F^{\rm T}E=\left(F^{\rm T}E \right)^{\rm T}$ from Theorem \ref{dmstar-fenjieshi}.
\end{proof}

Next, we study the relationship between the Dual-minus star partial order and the D-star partial  order.
\begin{theorem}
\label{D-DM-star-Th}
Let
$\widehat  E=E+\varepsilon E_0$,
$\widehat  F=F+\varepsilon F_0\in\mathbb{D}^{m\times n}$,
and $\widehat  E$ and $\widehat  F$ have the DMPGIs.
If $\widehat E\overset{\tiny\mbox{\rm D\!-}{*}}\leq\widehat F$, then
$\widehat E\overset{\tiny\mbox{\rm DM\!-}{*}}\leq\widehat F$.
\end{theorem}

\begin{proof}
Let
${\rk}\left( E \right)=r_e$, ${\rk}\left( F \right)=r_f$.
Since
 $\widehat E\overset{\tiny\mbox{\rm D\!-}*}\leq\widehat F$,
 then by Lemma \ref{PEIDSPO-Char-2-Th}
 we obtain  that
  $E\overset{*}\leq F$
  and
 the forms of $\widehat E$ and  $\widehat F$ are as in (\ref{PEIDSPO-Char-2}).
 Then $\widehat F-\widehat E$ has the DMPGI by Lemma \ref{Wang2021MAMT-Th2.1}.
 Therefore,  according to Definition
\ref{def-DMS},
we have $\widehat E\overset{\tiny\mbox{\rm DM\!-}*}\leq\widehat F$.
\end{proof}

\begin{theorem}
\label{nice---2}
Let $\widehat E=E+\varepsilon E_0$,
$\widehat F=F+\varepsilon F_0\in\mathbb{D}^{m\times n}$,
and $\widehat  E$ and $\widehat  F$ have DMPGIs. Then
$\widehat E\overset{\tiny\mbox{\rm  D\!-}\ast}\leq\widehat F$ if and only if
 $\widehat E\overset{\tiny\mbox{\rm  DM\!-}\ast}\leq\widehat F$ and
\begin{align}
\label{go-8}
 EE^\dagger\left(F_0-E_0\right)=\left(E^{\rm {\rm T}}\right)^\dagger E_0^{\rm T}\left(E-F\right), \
 \left(F_0-E_0\right)E^\dagger E=\left(E-F\right)E_0^{\rm T}\left(E^{\rm T}\right)^\dagger.
 \end{align}
\end{theorem}

\begin{proof}

 ``$\Leftarrow$''
\
Let $\widehat E=E+\varepsilon E_0$,
$\widehat F=F+\varepsilon F_0\in\mathbb{D}^{m\times n}$
and
$\widehat E\overset{\tiny\mbox{\rm  DM\!-}\ast}\leq\widehat F$.
Then the forms of
$\widehat E$ and $\widehat F$ are as in (\ref{DSPO-Char-2star}).
It is easy to check that
\begin{align}
\label{20231008-3}
EE^\dagger\left(F_0-E_0\right)
&
=
U\begin{pmatrix}
I_{{\rk}\left( E \right)}&O&O\\
O&O&O\\
O&O&O
\end{pmatrix}
\begin{pmatrix}
O&F_2-E_2&O\\
F_4-E_4&F_5&F_6\\
O&F_8&O\end{pmatrix}V^{\rm T}
=
U
\begin{pmatrix}
O&F_2-E_2&O\\
O&O&O\\
O&O&O\end{pmatrix}V^{\rm T},
\\
\nonumber
\left(E^{\rm {\rm T}}\right)^\dagger E_0^{\rm T}\left(E-F\right)&
=
U\begin{pmatrix}
D_1^{-1}&O&O\\
O&O&O\\
O&O&O
\end{pmatrix}
\begin{pmatrix}
E_1^{\rm T}&E_4^{\rm T}&E_7^{\rm T}\\
E_2^{\rm T}&O&O\\
E_3^{\rm T}&O&O
\end{pmatrix}
\begin{pmatrix}
O&O&O\\
O&-D_2&O\\
O&O&O
\end{pmatrix}V^{\rm T}=
U\begin{pmatrix}
O&-D_1^{-1}E_4^{\rm T}D_2&O\\
O&O&O\\
O&O&O
\end{pmatrix}V^{\rm T},
\end{align}
It follows from
$EE^\dagger\left(F_0-E_0\right)
=
\left(E^{\rm {\rm T}}\right)^\dagger E_0^{\rm T}\left(E-F\right)$
that  $F_2=E_2-D_1^{-1}E_4^{\rm T}D_2$.

 Since
\begin{align}
\label{20231008-4}
\left(F_0-E_0\right)E^\dagger E
&=
U\begin{pmatrix}
O&F_2-E_2&O\\
F_4-E_4&F_5&F_6\\
O&F_8&O\end{pmatrix}
\begin{pmatrix}
I_{{\rk}\left( E \right)}&O&O\\
O&O&O\\
O&O&O
\end{pmatrix} V^{\rm T}
=
U\begin{pmatrix}
O&O&O\\
F_4-E_4&O&O\\
O&O&O
\end{pmatrix}V^{\rm T},
\\
\nonumber
\left(E-F\right)E_0^{\rm T}\left(E^{\rm T}\right)^\dagger
&
=
U\begin{pmatrix}
O&O&O\\
O&-D_2&O\\
O&O&O
\end{pmatrix}
\begin{pmatrix}
E_1^{\rm T}&E_4^{\rm T}&E_7^{\rm T}\\
E_2^{\rm T}&O&O\\
E_3^{\rm T}&O&O
\end{pmatrix}
\begin{pmatrix}
D_1^{-1}&O&O\\
O&O&O\\
O&O&O
\end{pmatrix}V^{\rm T}
=U\begin{pmatrix}
O&O&O\\
-D_2E_2^{\rm T}D_1^{-1}&O&O\\
O&O&O
\end{pmatrix}V^{\rm T},
\end{align}
applying
$\left(F_0-E_0\right)E^\dagger E
=
\left(E-F\right)E_0^{\rm T}\left(E^{\rm T}\right)^\dagger$,
we have   $F_4=E_4-D_2E_2^{\rm T}D_1^{-1}$.

Substituting $F_4=E_4-D_2E_4D_1^{-1}$ and $F_2=E_2-D_1^{-1}E_2D_2$
into (\ref{DSPO-Char-2star}) gives (\ref{PEIDSPO-Char-2}).
Therefore,
$\widehat E\overset{\tiny\mbox{\rm  D \!-}*}\leq\widehat F$ holds.

``$\Rightarrow$''
\
Let $\widehat E\overset{\tiny\mbox{\rm  D \!-}*}\leq\widehat F$.
It follows from Theorem \ref{D-DM-star-Th}
that
$\widehat E\overset{\tiny\mbox{\rm DM\!-}*}\leq\widehat F$
and
  the the forms of
$\widehat E$ and $\widehat F$ are as in (\ref{PEIDSPO-Char-2}).
 Then  (\ref{go-8}) holds.
\end{proof}

Next,
we consider the relationship between the Dual-minus star partial order and the P-star partial  order.

\begin{theorem}
\label{20231008-5}
Let
$\widehat  E=E+\varepsilon E_0$,
$\widehat  F=F+\varepsilon F_0\in\mathbb{D}^{m\times n}$,
 and $\widehat  E$ and $\widehat  F$ have DMPGIs.
If $\widehat E\overset{\tiny\mbox{\rm P\!-}{*}}\leq\widehat F$, then
$\widehat E\overset{\tiny\mbox{\rm DM\!-}{*}}\leq\widehat F$.
\end{theorem}

\begin{proof}
 Let $\widehat  E$, $\widehat  F\in\mathbb{D}^{m\times n}$,
${\rk}\left( E \right)=r_e$, ${\rk}\left( F \right)=r_f$, $\widehat  E$ and $\widehat  F$ have the DMPGIs, and
 $\widehat E\overset{\tiny\mbox{\rm P\!-}*}\leq\widehat F$.
 Then we obtain  that
  $E\overset{*}\leq F$
  and
 the forms of $\widehat E$ and $\widehat F$ are as in (\ref{peiDPSO-Char-1}) by Lemma \ref{peiDPSO-Char-1-Th}.
Then $\widehat F-\widehat E$ has the DMPGI by Lemma \ref{Wang2021MAMT-Th2.1}.
 Therefore,  according to Definition
\ref{def-DMS},
we have $\widehat E\overset{\tiny\mbox{\rm DM\!-}*}\leq\widehat F$.
\end{proof}

\begin{theorem}
\label{nice---3}
Let $\widehat E=E+\varepsilon E_0$,
$\widehat F=F+\varepsilon F_0\in\mathbb{D}^{m\times n}$,
and $\widehat  E$ and $\widehat  F$ has the DMPGIs.
Then $\widehat E\overset{\tiny\mbox{\rm  P\!-}{*}}\leq\widehat F$ if and only if
 $\widehat E\overset{\tiny\mbox{\rm  DM\!-}\ast}\leq\widehat F$
 and
 \begin{align}
\label{go-10}
EE^\dagger\left(F_0-E_0\right)=O, \
 \left(F_0-E_0\right)E^\dagger E=O.
\end{align}
\end{theorem}

\begin{proof}
``$\Leftarrow$'' \
Let $\widehat E=E+\varepsilon E_0$,
 $\widehat F=F+\varepsilon F_0\in\mathbb{D}^{m\times n}$ 
and
 $\widehat E\overset{\tiny\mbox{\rm DM\!-}*}\leq\widehat F$.
Then the forms of $\widehat E$ and $\widehat F$ are as in
 (\ref{DSPO-Char-2star}).
Since
 $EE^\dagger\left(F_0-E_0\right)=O$
 and
 $\left(F_0-E_0\right)E^\dagger E=O$,
 we get  $F_4=E_4$ and   $F_2=E_2$.
Therefore,
$\widehat E\overset{\tiny\mbox{\rm  P \!-}*}\leq\widehat F$ holds.

``$\Rightarrow$'' \
Let $\widehat E\overset{\tiny\mbox{\rm  P \!-}*}\leq\widehat F$.
It follows from Theorem \ref{20231008-5}
that
$\widehat E\overset{\tiny\mbox{\rm DM\!-}*}\leq\widehat F$
and
  the the forms of
$\widehat E$ and $\widehat F$ are as in (\ref{peiDPSO-Char-1}).
It is easy to check that
 $EE^\dagger\left(F_0-E_0\right)=O$
 and
 $\left(F_0-E_0\right)E^\dagger E=O$.
Therefore, (\ref{go-10}) holds.
\end{proof}

At last, we show
the relationships among
the Dual-minus partial order,
the Dual-minus sharp partial order,
  the D-sharp partial order
and the G-sharp partial order
in Figure 1,
and the relationships among
the Dual-minus partial order,
the Dual-minus star partial order,
  the D-star partial order
and the P-star partial order
in Figure 2.

\begin{center}
\begin{tikzpicture}
    \node (E) at (0,0) {$\widehat E\overset{\tiny\mbox{\rm D\!}}\leq\widehat F$};
    \node (F) at (-2,-3.5) {$\widehat E\overset{\tiny\mbox{\rm DM\!-}{\#}}\leq\widehat F$};
    \node (G) at (2,-3.5) {$\widehat E\overset{\tiny\mbox{\rm  D\!-}\#}\leq\widehat F$};

    \draw[->] (-0.6,-0.6) -- (-2.02,-3.0) node[draw,midway, above, sloped,yshift=4pt] {(\ref{20231010-1})};
    \draw[->] (0.6,-0.6) -- (2.02,-3.0) node[draw,midway, above, sloped,yshift=4pt] {(\ref{go2})};
    \draw[->] (-1.0,-3.6) -- (1.0,-3.6) node[draw,midway, above, sloped,yshift=-18pt] {(\ref{perfect-1})};

    \draw[->] (-1.8,-3.0)--(-0.4,-0.6);
    \draw[->] (1.8,-3.0)--(0.4,-0.6);
    \draw[->] (1.0,-3.4) -- (-1.0,-3.4);
\end{tikzpicture}  \qquad
\begin{tikzpicture}
    \node (E) at (0,0) {$\widehat E\overset{\tiny\mbox{\rm D\!}}\leq\widehat F$};
    \node (F) at (-2,-3.5) {$\widehat E\overset{\tiny\mbox{\rm DM\!-}{\#}}\leq\widehat F$};
    \node (G) at (2,-3.5) {$\widehat E\overset{\tiny\mbox{\rm  G\!-}\#}\leq\widehat F$};

    \draw[->] (-0.6,-0.6) -- (-2.02,-3.0) node[draw,midway, above, sloped,yshift=4pt] {(\ref{20231010-1})};
    \draw[->] (0.6,-0.6) -- (2.02,-3.0) node[draw,midway, above, sloped,yshift=4pt] {(\ref{gogo}};
    \draw[->] (-1,-3.6) -- (1,-3.6) node[draw,midway, above, sloped,yshift=-18pt] {(\ref{perfect-4})};

    \draw[->] (-1.8,-3.0)--(-0.4,-0.6);
    \draw[->] (1.8,-3.0)--(0.4,-0.6);
    \draw[->] (1,-3.4) -- (-1,-3.4);
\end{tikzpicture}
\\
Figure 1:\ The relationships among
the Dual-minus partial order,
the Dual-minus sharp partial order,
  the D-sharp partial order
and the G-sharp partial order
\end{center}

\begin{center}
\begin{tikzpicture}
    \node (E) at (0,0) {$\widehat E\overset{\tiny\mbox{\rm D\!}}\leq\widehat F$};
    \node (F) at (-2,-3.5) {$\widehat E\overset{\tiny\mbox{\rm DM\!-}{*}}\leq\widehat F$};
    \node (G) at (2,-3.5) {$\widehat E\overset{\tiny\mbox{\rm  D\!-}\ast}\leq\widehat F$};

    \draw[->] (-0.6,-0.6) -- (-2.02,-3.0) node[draw,midway, above, sloped,yshift=4pt] {(\ref{perfect-6})};
    \draw[->] (0.6,-0.6) -- (2.02,-3.0) node[draw,midway, above, sloped,yshift=4pt] {(\ref{go})};
    \draw[->] (-1,-3.6) -- (1,-3.6) node[draw,midway, above, sloped,yshift=-18pt] {(\ref{go-8})};

    \draw[->] (-1.8,-3.0)--(-0.4,-0.6);
    \draw[->] (1.8,-3.0)--(0.4,-0.6);
    \draw[->] (1,-3.4) -- (-1,-3.4);
\end{tikzpicture}  \qquad
\begin{tikzpicture}
    \node (E) at (0,0) {$\widehat E\overset{\tiny\mbox{\rm D\!}}\leq\widehat F$};
    \node (F) at (-2,-3.5) {$\widehat E\overset{\tiny\mbox{\rm DM\!-}{*}}\leq\widehat F$};
    \node (G) at (2,-3.5) {$\widehat E\overset{\tiny\mbox{\rm  P\!-}\ast}\leq\widehat F$};

    \draw[->] (-0.6,-0.6) -- (-2.02,-3.0) node[draw,midway, above, sloped,yshift=4pt] {(\ref{perfect-6})};
    \draw[->] (0.6,-0.6) -- (2.02,-3.0) node[draw,midway, above, sloped,yshift=4pt] {(\ref{go1})};
    \draw[->] (-1,-3.6) -- (1,-3.6) node[draw,midway, above, sloped,yshift=-18pt] {(\ref{go-10})};

    \draw[->] (-1.8,-3.0)--(-0.4,-0.6);
    \draw[->] (1.8,-3.0)--(0.4,-0.6);
    \draw[->] (1,-3.4) -- (-1,-3.4);
\end{tikzpicture}
\\
Figure 2:\
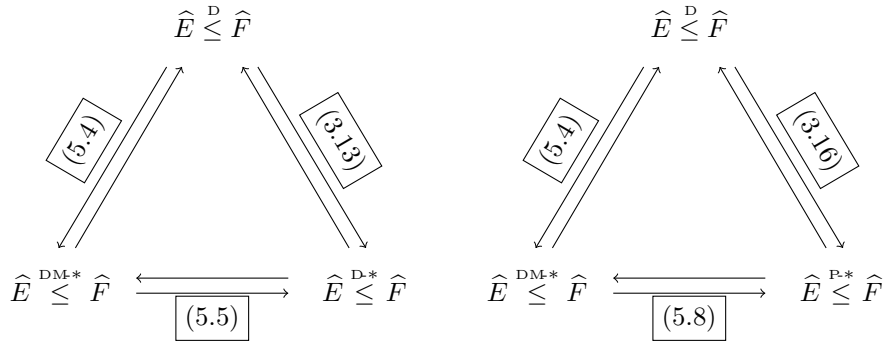
The relationships among
the Dual-minus partial order,
the Dual-minus star partial order,
and the D-star partial order
and the P-star partial order
\end{center}

%



\end{document}